\documentclass{elsarticle}
\usepackage{amsmath,amsfonts,amssymb,amscd,amsthm,amsbsy,epsf,calc,graphicx,color,geometry,natbib}
\newtheorem{thm}{Theorem}[section]
\newtheorem{lem}[thm]{Lemma}
\newtheorem{cor}[thm]{Corollary}
\newtheorem{prop}[thm]{Proposition}

\newtheorem*{THM}{Theorem}

\newtheorem{DEF}[thm]{Definition}
\theoremstyle{remark}                  
\newtheorem{rem}[thm]{Remark}

\newcommand{\BlueComment}[1]{\color{black}{#1}\color{black}}

\begin{document}

\begin{abstract}
We introduce a notion of non-local almost minimal boundaries similar to that introduced by Almgren in geometric measure theory. Extending methods developed recently for non-local minimal \BlueComment{boundaries } we prove that flat non-local almost minimal boundaries are smooth. This can be viewed as a non-local version of the Almgren-De Giorgi-Tamanini regularity theory. The main result has several applications, among these $C^{1,\alpha}$ regularity for sets with prescribed non-local mean curvature in $L^p$ and regularity of solutions to non-local obstacle problems.
\end{abstract}


\title{Regularity for non-local almost minimal boundaries and applications}

\author{M. Cristina Caputo}

\ead{caputo@math.utexas.edu}

\author{Nestor Guillen \corref{cor1}}

\ead{nguillen@math.utexas.edu}

\cortext[cor1]{Nestor Guillen is partially supported by NSF grant DMS-0654267}

\address{Department of Mathematics. University of Texas at Austin, Austin, TX 78712, USA}

\maketitle


\section{Introduction}

In this article we introduce a notion of non-local almost minimal \BlueComment{boundaries } and study their regularity near flat points \BlueComment{as well as their singularities}. In rough terms,  a non-local minimal \BlueComment{boundary corresponds to } a set whose characteristic function minimizes a \BlueComment{non-local energy functional, namely a } fractional Sobolev norm (\BlueComment{specifically }$H^{s/2}$, $s<1$); they generalize classical minimal surfaces \BlueComment{of codimension one}, in particular, it has been observed that when $s=1$ they are nothing but minimal surfaces.

Non-local minimal surfaces appeared recently in the study of phase transitions with long range interactions, as equilibrium configurations. This was shown in \cite{CS10}, via a level set type scheme (see also \cite{I09}). The regularity theory of non-local minimal surfaces was developed in \cite{CRS10}, there it was shown that such surfaces are smooth except a singular set of at most dimension $n-2$ ($n$ is the space dimension), we follow closely the methods developed there. \\

Our main result states that if an \emph{almost minimal} \BlueComment{boundary} (\BlueComment{for a non-local energy functional}) is flat enough near a point, then in a neighborhood of that point it is smooth, i.e. it has a continuously changing normal. Namely:

\begin{thm}\label{thm main}
Suppose $E$ is $(J_s,\rho,\delta)$-minimal in $B_1$, where $\rho$ satisfies assumptions A1 through A3 (Section 2). There exists $\delta_0=\delta_0(n,s,\rho)$ such that if $E$ is $\delta_0$-flat in $B_1$, then $\partial E$ is $C^{1}$ in $B_{1/2}$.\\
\end{thm}

\BlueComment{
Moreover, we show that the set of points where $\partial E$ is not $C^1$ is at most $n-2$ dimensional.
\begin{thm}\label{thm main2}
If $E$ is $(J_s,\rho,d)$-minimal in $\Omega$, then the Hausdorff dimension of the singular set $\Sigma_E \subset \partial E \cap \Omega$ is at most $n-2$.
\end{thm}}

For the definition of a $(J_s,\rho,\delta_0)$-minimal set and $\delta$-flatness, see Section 2.\\

The classical theory of almost minimal boundaries says that a set $E$ whose perimeter does not decrease too much under small perturbations has a smooth boundary, except perhaps for a lower dimensional set of singularities. This is known as the Almgren-De Giorgi-Tamanini regularity theory (see Chapter 4 of \cite{A76} or \cite{T84}, also the classical reference \cite{G84}). The term ``almost'' refers to the fact that the energy does not decrease too much under small perturbations. In the present paper we prove that a similar statement holds if instead of the perimeter we consider the energy associated to $H^{s/2}$, which is a non-local functional. It is worth mentioning also that the regularity theory for fully non-linear elliptic equations has been developed recently for non-local operators \cite{CS09}.\\

The proof of the above theorem  follows substantially the techniques in~\cite{CRS10} which are based on Savin's proof of regularity of minimal surfaces and level sets in phase transitions (see \cite{S07} and \cite{S09})\BlueComment{, they are also closely related to work of Caffarelli and C\'ordoba \cite{CC93}.}\\

To prove regularity of non-local minimal surfaces there are three basic steps \BlueComment{\cite{CC93,G84}}:  1) minimal surfaces have no cusps in measure (volume density estimates); 2) minimizers satisfy the minimal surface equation in the viscosity sense (Euler-Lagrange equation); 3) existence of tangent cones (monotonicity formula).\\

The main obstacles in carrying out this procedure for almost minimizers are extending the volume \BlueComment{density }estimates and finding a substitute for the Euler-Lagrange equation. We overcome these first by extending De Giorgi's differential inequality argument to non-local almost minimizers, complementing the discrete iteration argument used in \cite{CRS10} to get uniform volume density; secondly we show that the almost minimality of a set forces it to solve a variational inequality, this variational inequality (see Section 5) is the right substitute for the minimal surface equation. \\

Classically, almost minimal boundaries appear as solutions of a large class of geometric problems: prescribed mean curvature, obstacle problems, phase transitions (for instance, their regularity theory plays an important role in Luckhaus' work on the Stefan problem with Gibbs-Thomson law \cite{L90}). Accordingly, analogous applications arise in the non-local setting (e.g. prescribed non-local mean curvature). In forthcoming work, we shall use this regularity result to study the non-local mean curvature flow in the spirit of \cite{LS95}. The present result will also be used in the future to study the boundary behavior of non-local minimal surfaces.

\bigskip
\textbf{Remark.} As $s \to 1$ the energy $J_s$ (see Section 2) converges to the perimeter. Thus $E$ being $(J_1,\rho,\delta)$-minimal \BlueComment{means } $E$ is an almost minimal boundary in the classical sense. \BlueComment{It is an interesting question whether our estimates are uniform, giving perhaps a new proof of the Almgren-De Giorgi-Tamanini regularity theory for almost minimal boundaries. This issue will be addressed elsewhere.}\\

The article is organized as follows: In sections 2 and 3 we introduce the basic definitions and notation, in particular \BlueComment{the notion of almost minimal sets for non-local energy functionals}, and present several examples that justify this definition. In section 4 we prove the uniform volume estimate, which extends that of \cite{CRS10} to the almost minimal case.

The key steps of the proof Theorem \ref{thm main} are carried out in sections 5 and 6, where we prove the Euler-Lagrange inequalities (Theorem \ref{thm second Euler Lagrange}) and the Partial Harnack estimate (Theorem \ref{thm Harnack estimate}) \BlueComment{these imply an } improvement of flatness via compactness and blow up techniques. The strategy of the proof is discussed at greater length at the beginning of each of these sections.

Finally, in section 7 we prove a monotonicity formula and state without proof several propositions that imply by known methods the estimate on the singular set \BlueComment{(Theorem \ref{thm main2})}.

\section{Definitions and notation}

First, we define what it means for a set to be $\delta$-flat.

\begin{DEF}
A set $E \subset \mathbb{R}^n$ is said to be $\delta$-flat in $B_r(x_0)$ if there \BlueComment{is a } unit vector $e$ such that 
\begin{equation*}
\partial E \cap B_r(x_0) \subset \{ |(x-x_0) \cdot e| \leq \delta r \}.
\end{equation*}
\end{DEF}
\begin{DEF}
Given any two measurable sets $A,B \subset \mathbb{R}^n$, and $s \in (0,1)$, we shall write	
\begin{equation*}
L_s(A,B):=\int_{\mathbb{R}^n}\int_{\mathbb{R}^n}(1-s)\frac{\chi_A(x)\chi_B(y)}{|x-y|^{n+s}}dxdy.
\end{equation*}
\end{DEF}
Note that the integral above might take the value $+\infty$. It \BlueComment{easy } to see that if $E \subset \mathbb{R}^n$ \BlueComment{then}
\begin{equation*}
\|\chi_E\|_{H^{s/2}}^2 = c_{n,s}L_s(E,E^c).
\end{equation*}

\BlueComment{This quantity will play the part of a perimeter functional in all that follows, heuristically one can think about it as a sort of ``interfacial energy'', just as the perimeter measures the amount of interaction between different states near the interface (see \cite{CS10}) and the references therein}. With this in mind, given a domain $\Omega$ we consider a ``localized energy'' functional which measures the contribution of $\Omega$ to the $H^{s/2}$ norm of $\chi_E$.

\begin{DEF}
Let $n \geq 2$ and let $\Omega \subset \mathbb{R}^n$ be a bounded domain with Lipschitz boundary. Given $E \subset \Omega$, we shall denote by $J_s(E;\Omega)$ the quantity 
\begin{equation}
J_s(E;\Omega):=L_s(E \cap \Omega,E^c) + L_s(E \setminus \Omega, E^c \cap \Omega).
\end{equation}

\end{DEF}

We will from now on drop the $s$ subscript from $L_s(.,.)$ without causing too much confusion. We are ready to define the notion of almost minimality for the functional $J_s$, it extends the notion introduced by Almgren for the perimeter (see \cite{A76}).

\begin{DEF}

Let $s\in (0,1)$, $\delta>0$ and $\rho: (0,\delta) \to \mathbb{R}^+$. We say that a measurable set $E$ is $(J_s,\rho,\delta)$-minimal in $\Omega$ if for any $x_0 \in \partial E$, any set $F$ \BlueComment{and any $r$ such that $r<\min(\delta,d(x_0,\partial \Omega))$ and $F \Delta E \subset B_r(x_0)$ }we have
\begin{equation}
J_s(E;\Omega)\leq J_s(F;\Omega)+\rho(r)r^{n-s}.
\end{equation}

\end{DEF}

The assumptions we impose on the modulus of continuity $\rho$ are \BlueComment{standard, they can be found in  Almgren's Memoir \cite{A76} p. 96a and Tamanini's monograph \cite{T84} p. 9. They are as follows}: \\

A1) $\rho:(0,\delta) \to \mathbb{R}^+$ is a non decreasing, bounded function.

A2) $\rho(t)=o(1)$.

\BlueComment{A3) The function $t^{-s}\rho(t)$ is non increasing in $(0,\delta)$  and for some $m>n+s$ we have
\begin{equation}\label{eqn A3}
\int_0^\delta z^{-1}\left (\rho(z) \right )^{\frac{1}{m}}dz <+\infty.
\end{equation}

We emphasize that these are entirely analogous to the assumptions made for hypersurfaces of almost minimal perimeter. In particular, for any $C>0$ and $\alpha\in (0,s]$ the assumptions above are always satisfied by $\rho(t)= Ct^\alpha$. If a $\rho$ does not satisfy the first part of A3, this means $\rho(t) \leq Ct^s$ for all small enough $t$, thus any $(J_s,\rho,\delta)$-minimal boundary is also $(J_s,Ct^s,\delta)$-minimal, so our theory also applies to this case\footnote{One expects that for $\rho(t)=o(t^s)$ almost minimality reduces to minimality, see remark preceding Corollary \ref{cor pointwise Euler Lagrange}.}. The following auxiliary function will play a central role in our theory.\\

Remark. Throughout this work, we will implicitly assume $\delta \geq 1$ to simplify things. All arguments hold for arbitrary small $\delta>0$ with few modifications (and absolute constants will depend on $\delta$, of course). Moreover, by rescaling one can carry the result for $\delta>1$ to the general case.

\begin{DEF}\label{def rho hat}
Given $\rho$ satisfying the assumptions above, we define an auxiliary function $\hat\rho$.
\begin{equation}\label{eqn rho}
\hat\rho: (0,\delta)\to \mathbb{R}, \;\;\; \hat \rho(t):= \frac{s}{m}\int_0^t z^{-1}\rho(z)^{\frac{1}{m}}dz 
\end{equation}
Where $m$ is the same constant greater than $n+s$ which comes as part of assumption A3.
\end{DEF}

Remarks. Note that by A3 we have $\hat\rho(t) \to 0$ as $t \to 0$, and that $\hat\rho$ is a non-decreasing continuous function. As we will see in Section 6, the $C^1$ modulus of continuity in Theorem \ref{thm main} is actually $C^{1,\hat\rho}$. In light of this, when $\rho\leq Ct^\alpha$, Theorem \ref{thm main} says that $E$ has a $C^{1,\alpha'}$ boundary in $B_{1/2}$, where $\alpha'=\alpha/m$. The following estimate for $\hat\rho$ will be very important at several stages of the proof.

\begin{lem}\label{lem rho doubling}
Let $\rho$ and $\hat\rho$ be as above. Then we have the bound	
\begin{equation*}
\rho(t\epsilon)^{\frac{1}{m}}	\leq \hat \rho(t)\;\;\forall \;\epsilon \in (0,1),\; t \in (0,\delta).
\end{equation*}	
\end{lem}
\begin{proof}
By definition, 
\begin{equation*}
\frac{\hat\rho(t)}{\rho(t \epsilon)^{\frac{1}{m}}}=\frac{s}{m}\int_0^{t}\frac{\rho(z)^{\frac{1}{m}}}{z\rho(t\epsilon)^{\frac{1}{m}}}dz.
\end{equation*}	
Since the integrand is positive and $\epsilon \in (0,1)$,
\begin{equation*}
\frac{\hat\rho(t)}{\rho(t \epsilon)^{\frac{1}{m}}}\geq \frac{s}{m}\int_0^{t\epsilon}\frac{\rho(z)^{\frac{1}{m}}}{ z\rho(t\epsilon)^{\frac{1}{m}}}dz.
\end{equation*}	
Thanks to condition A3, we conclude
\begin{equation*}
\int_0^{t\epsilon}\frac{\rho(z)^{\frac{1}{m}}}{ z\rho(t\epsilon)^{\frac{1}{m}}}dz\geq\int_0^{t\epsilon}\left (\frac{z}{t\epsilon}\right )^{\frac{s}{m}}\frac{1}{z}dz.
\end{equation*}	
Thus,
\begin{equation*}
\frac{\hat\rho(t)}{\rho(t \epsilon)^{\frac{1}{m}}}\geq 	\frac{s}{m}(\epsilon t)^{-\frac{s}{m}}\int_0^{t\epsilon} z^{\frac{s}{m}-1}dz=(\epsilon t)^{-\frac{s}{m}}(t\epsilon)^{\frac{s}{m}}=1
\end{equation*}
Which shows that $\rho(t\epsilon)^{\frac{1}{m}} \leq \hat\rho(t)$ for all $\epsilon\in(0,1)$.
\end{proof}
}
To finish these preliminaries, we observe that, similarly to the case of a \BlueComment{non-local } minimal boundary, being $(J_s,\rho,\delta)$-minimal is equivalent to having the following two properties:\\

\emph{Supersolution property}: For any $x_0 \in \partial E$ and $A \subset E^c \cap B_r(x_0)$, $ r<\min(\delta,d(x_0,\partial \Omega))$ we have

\begin{equation}\label{eqn supersolution}
L(A,E)-L(A,E^c\setminus A) \leq \rho(r)r^{n-s}.
\end{equation}

\bigskip

\emph{Subsolution property}: For any $x_0 \in \partial E$ and $A \subset E \cap B_r(x_0)$, $ r<\min(\delta,d(x_0,\partial \Omega))$ we have

\begin{equation}\label{eqn subsolution}
L(A,E\setminus A)-L(A,E^c) \geq -\rho(r)r^{n-s}.
\end{equation}

\bigskip

To see the above equivalence, it is enough to check that for any set $F$, such that $F \Delta E \subset \subset \Omega$,
\begin{equation*}
J_s(F;\Omega)-J_s(E;\Omega) = L(A^-,A^+) 
\end{equation*}
\begin{equation*}
+\left ( L(A^-,E\setminus A^-)-L(A^-,E^c) \right )  - \left ( L(A^+,E)-L(A^+,E^c \setminus A^+) \right ),
\end{equation*}
\begin{equation*}
\mbox{ where } A^+ =  F \setminus E,\; A^- = E \setminus F.
\end{equation*}

In particular, we observe that
\begin{equation}\label{eqn J_B=J_Omega}
J_s(E;\Omega_1)-J_s(F;\Omega_1)=J_s(E;\Omega_2)-J_s(F;\Omega_2),
\end{equation}
\begin{equation*}
\mbox{whenever }  E \Delta F \subset \Omega_2 \subset \Omega_1.
\end{equation*}

Due to the \BlueComment{this, when the set $\Omega$ is clear from the context we may write } $J_s(E)$ instead of $J_s(E;\Omega)$.


\section{Examples}

Boundaries of almost minimal perimeter appear as solutions of various geometrical problems in the calculus of variations. In this section we discuss some of the analogous problems for the non-local energy, our presentation extends the discussion done \BlueComment{for the perimeter by Tamanini } \cite{T84}.\\

\textbf{Example 1. [Minimal boundaries]} Trivially, any non-local minimal \BlueComment{boundary } (i.e. a minimizer of $J_s$ among all sets that are prescribed outside $\Omega$) is \BlueComment{a $(J_s,0,\delta)$-minimal set}.

\bigskip

\textbf{Example 2. [Smooth Boundaries]} As it is the case for the perimeter, any set $E$ with a smooth boundary is $(J_s,\rho,\delta)$-minimal for $\rho$ and $\delta$ depending on the continuity of the normal to \BlueComment{$\partial E$.}

In particular, let $E \subset \mathbb{R}^n$ be the subgraph of a smooth function $u:\mathbb{R}^{n-1} \to \mathbb{R}$ which is globally bounded \BlueComment{and $C^{1,\alpha}$}, then $E$ is a $(J_s,Ct^\alpha,\delta)$-minimal set in any compact set of $\mathbb{R}^n$, for some $\alpha \in (0,s)$. Of course the constants $C>0$ and $\delta>0$ depend on $u$ and the compact set in question.

Let us prove this. First, note we may assume without loss of generality that the plane $\{x_n=0\}$ is tangent to $\partial E$ at $0$, so that $u(0)=0$, $\nabla u(0)=0$. Let $F$ be a measurable set such that $F \Delta E$ is contained in a ball of radius $r$ centered at some point, which we may also assume to be the origin. Then, thanks to (\ref{eqn J_B=J_Omega}) we have that 
\begin{equation}\label{eqn J_B=J_Omega two}
J_s(E;\Omega)-J_s(F;\Omega)=J_{s}(E;B_r)-J_{s}(F;B_r).
\end{equation}

Regardless of what the compact $\Omega$ is, as long as $B_r \subset \Omega$. On the other hand, under the regularity assumption of $u$ it is straightforward to check  $\partial E \cap B_r \subset \{ |x \cdot e_n | \leq Cr^{1+\alpha} \}$.\BlueComment{ In parallel to the perimeter, one may use a monotonicity formula (see Section 7) } to show that for any $F$ and small enough $r$ we have
\begin{equation*}
J_s(F;B_r) \geq \BlueComment{(1-Cr^\alpha)J_s(H;B_1)r^{n-s}},
\end{equation*}
where $H$ is the horizontal half-space going through the origin. Additionally, by the regularity of $u$ it can be checked (with a tedious but simple computation) that
\begin{equation*}
J_s(E;B_r)\leq r^{n-s}J_s(H;B_1) \left ( 1+  O(r^{\alpha}) \right ),
\end{equation*}
the last two inequalities imply that $E$ is then $(J_s,Ct^\alpha,\delta)$-minimal for some $\delta>0$ and $C>0$.\\

Let us sketch this last computation. It is convenient  to do the change of variables $x \to rx$, $E \to E_r$, we check easily that
\begin{equation*}
J_s(E;B_r)=r^{n-s}J_s(E_r;B_1),
\end{equation*}
where $E_r$ is now the subgraph of the function $u_r(x)=r^{-1}u(rx)$. To compute $J_s(E_r;B_1)$ we are going to change from the variables $(x,y)$ to the variables $(v,w)$, as follows:
\begin{equation}
x=T_r(v)=v+u_r(v')e_n, \; \; y=T_r(w)=w+u_r(w')e_n.
\end{equation}

In other words, $T_r$ flattens the boundary of the set $E_r$ down to the half space  $H=\{x_n \leq 0\}$. It can be easily checked that
\begin{equation*}
DT_r= I + \nabla u_r \otimes e_n, \; det(DT_r)=1.
\end{equation*}

Therefore
\begin{equation*}
J_s(E_r;B_1)=L(E_r \cap B_1,E_r^c)+L(E_r \setminus B_1, E_r^c \cap B_1)=
\end{equation*}

\begin{equation*}
=\int_{E_r^c} \int_{E_r \cap B_1}\frac{1}{|x-y|^{n+s}}dxdy+\int_{E_r^c \cap B_1} \int_{E_r \setminus B_1}\frac{1}{|x-y|^{n+s}}dxdy
\end{equation*}

\begin{equation*}
=\int_{H^c} \int_{H \cap T_r(B_1)} \frac{1}{|T_r(v)-T_r(w)|^{n+s}}dvdw+\int_{H^c \cap T_r(B_1)} \int_{H \setminus T_r(B_1)}\frac{1}{|T_r(v)-T_r(w)|^{n+s}}dvdw
\end{equation*}

\begin{equation*}
=\int_{H^c} \int_{H \cap T_r(B_1)} \frac{A(v,w,r)}{|v-w|^{n+s}}dvdw+\int_{H^c \cap T_r(B_1)} \int_{H \setminus T_r(B_1)}\frac{A(v,w,r)}{|v-w|^{n+s}}dvdw.
\end{equation*}

Where
\begin{equation*}
A(v,w,r)=\frac{|v-w|^{n+s}}{|T_r(v)-T_r(w)|^{n+s}}\leq 1+\BlueComment{C}\frac{|u_r(v')-u_r(w')|}{|v-w|}.
\end{equation*}

\BlueComment{For a $C>0$ depending on the bounds on $u$. } From here, it is not hard to see\BlueComment{ that}
\begin{equation*}
J_s(E;B_r)\leq r^{n-s}(1+Cr^\alpha)J_s(H;B_1),
\end{equation*} 
just as we wanted.\\

\textbf{Example 3. [Boundaries of least energy lying above an obstacle]} Let $E \subset \mathbb{R}^n$ solve the obstacle problem in $\Omega$ with respect to a set $L$, i.e. it minimizes $J_s(.)$ among all sets containing $L \cap \Omega$ and agreeing with a given fixed set outside $\Omega$ that contains $L \setminus \Omega$. The set $L$ is known as the obstacle and the set $\partial E \cap L$ is known as the contact set. 	If $L$ is $(J_s,\rho,\delta)$-minimal, then so is $E$.

To see this, let $F$ be a set such that $E \Delta F \subset B_r(x_0) \subset \Omega$. Since $F \cup L$ contains $L$ and agrees with $E$ outside $\Omega$ we see that $F \cup L$ is admissible, therefore by the assumption on $E$
\begin{equation*}
J_s(E)\leq J_s(F \cup L).
\end{equation*}
Now, one may check easily from the definition of $J_s$ that
\begin{equation*}
J_s(F \cup L) + J_s(F \cap L) \leq J_s(F) + J_s(L),
\end{equation*}
therefore
\begin{equation*}
J_s(E) -J_s(F) \leq J_s(L)-J_s(F \cap L) \leq \rho(r)r^{n-s},
\end{equation*}
where the last inequality follows from $E \Delta F \subset B_r(x_0)$ and the assumption on $L$.\\

\textbf{Remark}. Later we will see that near the contact set we can always apply Theorem \ref{thm main}, and thus $\partial E$ is always $C^1$ near the obstacle $L$ as long as $L$ is smooth enough.

\bigskip

\textbf{Example 4. [Boundaries with prescribed non-local mean curvature]} Another important class of examples of almost minimal boundaries consists of the minima of functionals of the form $\mbox{Per}(E)+\int_E \gamma(x)dx$. These appear for instance in phase transition problems where the mean curvature of the interface is related to the pressure or the temperature on the interface. It can be shown that the mean curvature of such minima (understood in a weak sense) is given by $\gamma(x)$, when instead of $\mbox{Per}(E)$ we consider the $J_s(E)$ energy what happens is that the non-local mean curvature of $E$ must agree \BlueComment{with $\gamma$.}

Consider the functional $\mathcal{F}(E)=J_s(E)+\int_E \gamma(x)dx$ with $\gamma(x) \in L^p$ ($p>n/s$). Suppose $E_0$ is a minimum of $\mathcal{F}$ among all sets $E$ that agree with $E_0$ outside $\Omega$, then $E_0$ is $(J_s,Ct^\alpha,\delta)$-minimal in $\Omega$, with $\alpha = s-\frac{n}{p}$.

To see $E_0$ is almost minimal, let $F$ be another set, if $E \Delta F$ lies in a small enough neighborhood of $E$, then $\mathcal{F}(E)\leq \mathcal{F}(F)$, which gives

\begin{equation*}
J_s(E) + \int_E\gamma(x)dx \leq J_s(F) +\int_F\gamma(x)dx
\end{equation*}

\begin{equation*}
J_s(E)-J_s(F) \leq \int_F\gamma(x)dx-\int_E\gamma(x)dx  \leq \int_{E \Delta F}|\gamma(x)|dx
\end{equation*}

by H\"older inequality, we get $J_s(E)-J_s(F) \leq |E \Delta F|^{1-\frac{1}{p}}\left ( \int_{E\Delta F} |\gamma(x)|^pdx\right )^{\frac{1}{p}}$.\\

Therefore, if $E \Delta F \subset B_r(x_0)$ for some point $x_0$ and some small $r$, we get

\begin{equation*}
J_s(E)-J_s(F) \leq C_n\rho_\gamma (r) r^{n-s}
\end{equation*}

where $\rho_\gamma(r) \leq  C_n r^{s-\frac{n}{p}} \| \gamma \|_{L^p}$. 

\section{\BlueComment{First estimates: almost minimal boundaries have no cusps.}}

The first step in the regularity theory of minimal surfaces consists in showing that if $x_0 \in \partial E$ and $E$ is minimal, then $\partial E$ separates any small ball centered at $x_0$ in two regions of comparable volume, so that at least heuristically $\partial E$ should not have cusps. In this section we prove such a volume estimate for non-local almost minimal boundaries.

\begin{prop}\label{prop density}
Let $E$ have the subsolution (supersolution) property in $\Omega$ and $x_0 \in \partial E \cap \Omega$, there are constants $r_0=r_0(n,s,\rho)$ and $c_0=c_0(n,s,\rho)$ such that for any $x \in \partial E$ and any $r \leq \min( r_0, d(x_0,\partial \Omega) )$ we have
\begin{equation*}
|E \cap B_r(x)| \geq c_0|B_r|, \;\; ( \mbox{resp. }  |E^c \cap B_r(x)| \geq c_0|B_r| \;).
\end{equation*}
In particular, if $E$ is $(J_s,\rho,\delta)$-minimal we have 
\begin{equation*}
c_0 \leq \frac{|E \cap B_r(x_0)|}{|B_r(x_0)|}\leq 1-c_0.
\end{equation*}
\end{prop}

\textbf{Remark}. The proof is done following De Giorgi's differential inequality argument, which is also used to prove uniform volume estimates for minimal surfaces (cf. section 2.8 in \cite{dG06} for further discussion), it exploits the two inequalities available to us, namely Sobolev's embedding and the subsolution (supersolution) property of $E$. Our proof is slightly different from the one used in \cite{CRS10} where a discrete version is used, both are nonetheless in the same spirit.

\begin{proof}
We will prove the first of the two inequalities, the second one is obtained \emph{mutatis mutandis}.  For $r\in(0,1)$, set $A_r:=E\cap B_r$ and consider the function
\begin{equation*}
f(r):=|A_r|,\qquad f'(r):=\mathcal {H}^{n-1}(E\cap \partial B_r).
\end{equation*}

\BlueComment{The identity for $f'$ is standard and follows from the co-area formula. }By the Sobolev inequality, 
\begin{equation*}
\|u\|_{L^p}\leq C_{n,s}\| u\|_{H^{\frac{s}{2}}},\quad p=\frac{2n}{n-s},
\end{equation*}
for $u=\chi_{E\cap B_r}$, we obtain
\begin{equation*}
f(r)^{\frac{n-s}{n}}\leq C L(A_r, A_r^c).
\end{equation*}

Since $E$ has the subsolution property, and $A_r \subset \Omega \cap E \cap B_r(x_0)$,
\begin{equation*}
 L(A_r,E^c) \leq L(A_r,E\setminus A_r)+\rho(r)r^{n-s},
\end{equation*}
so, 
\begin{equation*}
L(A_r,A_r^c)=L(A_r,E^c)+L(A_r,E\setminus A_r)
\end{equation*}
\begin{equation*}
\leq 2 L(A_r,E\setminus A_r)+\rho(r)r^{n-s}\leq 2 L(A_r, B_r(x_0)^c)+\rho(r)r^{n-s}
\end{equation*}

To estimate the $L(A_r,B_r(x_0)^c)$ term, we note that for all $x \in A_r$ we have the bound
\begin{equation*}
\int_{B_r^c}\frac{1}{|x-y|^{n+s}}dy \leq C_{n,s}\int_{r-|x|}^{\infty} \frac{1}{z^{n+s} }z^{n-1} dz \leq C_{n,s}(r-|x|)^{-s},
\end{equation*}
this yields to
\begin{equation*}
L(A_r, B_r^c)=\int\int\frac{\chi_{A_r}(x)\chi_{B_r^c}(y)}{|x-y|^{n-s}}dx dy\leq C_{n,s}\int_0^r f'(z)(r-z)^{-s} dz.
\end{equation*}

Hence, \BlueComment{we are led to the integro-differential inequality,}
\begin{equation*}\label{eqn density ide}
f(r)^{\frac{n-s}{n}}\leq C_{n,s}\left ( \int_0^r  f'(z)(r-z)^{-s}dz+\rho(r)r^{n-s} \right ).
\end{equation*}

Since we have $f'$ on the right hand side, we may integrate the last inequality in the $r$ variable on the interval $(0,t)$ and get
\begin{equation*}
\int_{0}^{t}\, f(r)^{\frac{n-s}{n}}dr \leq C_{n,s} \left ( \int_{0}^{t} \int_0^r  f'(z)(r-z)^{-s}dz dr + \int_{0}^{t}\rho(r)r^{n-s}dr \right ),
\end{equation*}
interchanging the order of integration,
\begin{equation*}
\int_{0}^{t} \int_0^r  f'(z)(r-z)^{-s}dz dr=\int_{0}^{t}f'(z) \int_{z}^{t}(r-z)^{-s}dr dz,
\end{equation*}
we get the integro-differential inequality,
\begin{equation}\label{eqn density ode}
\int_{0}^{t}\, f(r)^{\frac{n-s}{n}}dr \leq C_{n,s}t^{1-s} f(t)+\int_0^t\rho(r)r^{n-s}dr.
\end{equation}

Let $g(t)=C_0t^n-\int_0^t\rho(r)r^{n-1}dr$, with a conveniently chosen $C_0(n,s)$. It is not hard to check that then we have (with the same $C_{n,s}$ as in (\ref{eqn density ode})) 
\begin{equation}\label{eqn density ode two}
\int_0^tg(r)^{\frac{n-s}{n}}dr \geq 2C_{n,s}t^{1-s}g(t)+2\int_0^t\rho(r)r^{n-s}dr\;\;\BlueComment{\forall\; t\leq r_0(n,s,\rho)}
\end{equation}

\emph{Claim.} From the integro-differential inequalities (\ref{eqn density ode}) and (\ref{eqn density ode two}) \BlueComment{(or equivalently (\ref{eqn density ide}) and (\ref{eqn density ode two})) } we conclude that   $f(r) \geq g(r)$ for all $r$ such that $B_r(x_0) \subset \Omega$.\\ 

If this claim is true then we get the lower bound on $f$, since due to property \BlueComment{A3 } of $\rho$ there are $c_0=c_0(n,s,\rho)$ and $r_0=r_0(n,s,\rho)$ such that $g(r) \geq c_0r^n$ whenever $r \leq r_0$.\\

\emph{Proof of the claim.} We argue by contradiction as when dealing with differential inequalities: let $r_1$ be the supremum of the set $\{r:f(t)\geq g(t)\;\; \forall t \leq r \}$, assume that the claim was not true in which case $r_1<\min\{r_0,d(x_0,\partial \Omega)\}$, then for a sequence $h_k \to 0$, $h_k >0$ we have thanks to (\ref{eqn density ode}) and (\ref{eqn density ode two}) that
\begin{equation*}
g(r_1+h_k)>f(r_1+h_k)\geq Ct^{s-1}\int_{r_1}^{r_1+h_k}f(r)^{(n-s)/n}-\rho(r)r^{n-s}dr+2g(r_1)
\end{equation*}
\begin{equation*}
\Rightarrow \liminf g(r_1+h_k) > 2g(r_1),
\end{equation*}
which cannot be since $g$ is continuous, and this proves the claim.
\end{proof}

\BlueComment{
\begin{rem} We have in fact proved something stronger than what the statement of Proposition \ref{prop density} says: under the same hypothesis as before, there is a constant $r_0=r_0(n,s,\rho)$ such that for any $x \in \partial E$ and any $r \leq \min( r_0, d(x_0,\partial \Omega) )$ we have

\begin{equation*}
|E \cap B_r(x)|\geq c(n,s)r^n-\int_0^r\rho(z)z^{n-1}dz,
\end{equation*}
\begin{equation*}
( \mbox{resp. }  |E^c \cap B_r(x)| \geq c(n,s)r^n-\int_0^r\rho(r)z^{n-1}dz\;).	
\end{equation*}

\end{rem}

The volume density estimate implies (by a standard argument) that $L^1$ convergence of almost minimal boundaries $\{E_k\}$ guarantees their uniform convergence, this will be useful in the last section when we consider blow ups and tangent cones so we state it now as a corollary.

\begin{cor}
Let $\{ E_k\}$ be a sequence of $(J_s,\rho(.),\delta)$-minimal sets such that for some measurable set $E$ we have
\begin{equation*}
E_k \to  E \; \mbox { in } L^1_{loc}(\mathbb{R}^n)
\end{equation*}
then $E_k \to E$ uniformly in compact subsets of $\Omega$.
\end{cor}

We end this section with a compactness estimate , the proof is the same as that of Theorem 3.3 in \cite{CRS10} and it relies on the lower-semicontinuity of $J_s$ with respect to $L^1$ convergence.}

\begin{prop}\label{thm compactness of sets}
Suppose that for each $k >0$ the set $E_k$ is $(J_s,\rho_k,\delta)$-minimal in $\Omega$, assume also that $\rho_k(t) \to \rho(t)$ (pointwise) and $\{ \rho_k\}$ satisfies the assumptions A1-A4 uniformly in $k$. Then if $\{E_k\}$ converges in $L^1_{loc}(\Omega)$ to a set $E$, this set must be $(J_s,\rho,\delta)$-minimal in $\Omega$.
\end{prop}

\textbf{Remark}. Observe that whenever we have a sequence of almost minimal sets with $\rho_k$ as above we may use Sobolev embedding to get a subsequence that converges in $L^1$.

\section{Euler-Lagrange inequalities and partial Harnack}

In this section we will show how the almost minimality property of a set forces it to satisfy what we will call ``Euler-Lagrange inequalities''. These play the same role that the Euler-Lagrange equation (in the viscosity sense) plays in the theory of minimal surfaces, as shown first in \cite{CC93} for the perimeter and later in \cite{CRS10} for non-local energies. As in the latter work, the Euler-Lagrange inequalities will lead us here to a Harnack estimate whenever there is enough flatness.

We shall motivate the main idea by proving a maximum principle in a very simple case, after this, our main effort will be to see how close we are in the general case to this particular ideal situation and how can we adapt the argument below to a general setting. \\

\textbf{Maximum principle}. Suppose that $E$ is a \emph{minimizer} of $J_s(.)$ among all sets that agree with $E$ outside $\Omega$, and suppose that

\begin{equation*}
\{ x: x_n < 0 \} \setminus \Omega \subset E.
\end{equation*}
That is, suppose that $E$ outside $\Omega$ lies on the same side of some hyperplane, then it must lie on the same side of that hyperplane everywhere

\begin{equation*}
\{ x: x_n < 0\} \subset E.
\end{equation*}
\begin{proof} We argue by contradiction, to do this, let us slide from below a plane parallel to $\{x_n=0\}$ until we touch $\partial E$ inside $\Omega$. Since $E$ does not contain $\{ x_n < 0 \}$, there exists some $c>0$ such that $H= \{ x: x_n=-c\}$ touches from below $\partial E$ at some point $x_0 \in \Omega$. Then, for all small enough $\epsilon>0$ consider the sets:
\begin{eqnarray*}
A^-_\epsilon =& E^c \cap \{ x:x_n < -c+\epsilon\} \nonumber\\
A^+_\epsilon =& E^c \cap T_\epsilon(A^-_\epsilon)  \nonumber\\
A_\epsilon \;=& A^+_\epsilon \cup A^-_\epsilon \nonumber,
\end{eqnarray*}
here $T_\epsilon$ is the reflection with respect to the hyperplane $\{x:x_n=-c+\epsilon\}$. The set $A_\epsilon$ will be used to perturb $E$, note that because the hyperplane is touching $\partial E$ we have $|A_\epsilon|>0$ for all small $\epsilon$. Now, since $A_\epsilon \subset E^c$, and $E$ is a minimizer, we may use (\ref{eqn supersolution}) with $\rho \equiv 0$ to get
\begin{equation*}
L(A_\epsilon,E)-L(A_\epsilon,E^c\setminus A_\epsilon) \leq 0.
\end{equation*}
$T_\epsilon$ preserves the kernel $|x-y|^{-(n+s)}$ because it is an isometry,  therefore we can use it to compare above the opposite contributions of $E$ and $E^c\setminus A_\epsilon$ to the quantity above and get some cancellations, so let $F_\epsilon = T_\epsilon (E^c \setminus A_\epsilon)$. Clearly, $F_\epsilon \subset E$ and the inequality above is equivalent to
\begin{equation*}
L(A_\epsilon, E \setminus F_\epsilon) + \left [ L(A_\epsilon,F_\epsilon) - L(A_\epsilon, T(F_\epsilon) ) \right ]  \leq 0.
\end{equation*}
Moreover, $A_\epsilon$ can be decomposed as $S_\epsilon \cup D_\epsilon$, where $T(S_\epsilon)=S_\epsilon$ and $D_\epsilon \subset A^-_\epsilon$, so we can rewrite things one more time:
\begin{equation*}
L(A_\epsilon, E \setminus F_\epsilon) + \left [ L(S_\epsilon,F_\epsilon) - L(S_\epsilon, T(F_\epsilon) ) \right ] + \left [ L(D_\epsilon,F_\epsilon) - L(D_\epsilon, T(F_\epsilon) ) \right ] \leq 0,
\end{equation*}
by the symmetry of $S_\epsilon$ we have $L(S_\epsilon,F_\epsilon)=L(S_\epsilon,T(F_\epsilon))$.  On the other hand, because $F_\epsilon,D_\epsilon \subset \{x_n < -c+\epsilon\}$ it can be checked that $L(D_\epsilon,F_\epsilon) \geq L(D_\epsilon,T(F_\epsilon))$, thus
\begin{equation*}
L(A_\epsilon, E \setminus F_\epsilon)\leq 0 \Rightarrow L(A_\epsilon, E \setminus F_\epsilon) = 0.
\end{equation*}
 Since $|E \setminus F_\epsilon | >0$ for all small enough $\epsilon$, this means that $|A_\epsilon| = 0$,  a contradiction. So it must be that $\{ x: x_n\leq 0\} \subset E$ as we wanted to prove.
\end{proof}
This argument tells us that if $E$ is a minimizer in say $B_1(0)$ and $E \setminus B_1(0)$ is trapped between two hyperplanes: $E \subset \{ -a<x_n<-b\}$ then the same holds inside $B_1(0)$. It suggests we could do a similar argument if instead of sliding a hyperplane we could slide a very flat smooth surface from below (say a large ball or flat parabola); using such surfaces as barriers we may show arguing again as above that if $E$ is flat enough in a cylinder, then in a smaller cylinder $E$ cannot stretch too much vertically; this is part of the intuition behind the proofs of the Euler-Lagrange equation and the partial Harnack estimate of this section. First we recall the Euler Lagrange equation proved in \cite{CRS10} (also Theorem 5.1 there):


\begin{thm}\label{thm first Euler Lagrange}

Let $E$ be a minimizer in $\Omega$, suppose that $0 \in \partial E$ and that $E \cap \Omega$ contains the ball $B_{R}(-Re_n)$, $R\geq 1$.  Then there exist  vanishing sequences $\delta_k,\epsilon_k$ (with $\delta_k<<\epsilon_k$) and sets $\hat{A}_{k}$, with $\hat{A}_{k} \subset B_{\delta_k}(0) \cap E^c$, such that
\begin{equation*}
L(\hat{A}_{k},E \setminus B_{\delta_k}) - L(\hat{A}_{k},E^c \setminus B_{\delta_k}) \leq (C\delta_k^{1-s}+\epsilon_k^\eta)|\hat{A}_{k}|,\; \mbox{ for some } \eta>0.
\end{equation*}
Moreover,
\begin{equation*}
\int_{\mathbb{R}^n}\frac{\chi_E(x)-\chi_{E^c}(x)}{|x|^{n+s}}dx\leq 0.
\end{equation*}
\end{thm}

\bigskip

As we discussed previously, the Euler-Lagrange equation implies the partial Harnack inequality for flat minimizers. The same perturbation argument provides a similar estimate for almost minimal sets which comes with an extra term of order $\rho(\epsilon)\epsilon^{n-s}$.\\

\BlueComment{\textbf{Remarks}. 	We are going to do a slightly more general version of this estimate, which takes into account the deviation from minimality. Firstly, we will need to say something more about the sequence $\epsilon_k \to 0$ along which the estimates hold, in some sense, we will prove that in fact these inequalities hold for \emph{any} $\epsilon$ small enough (see Theorem \ref{thm second Euler Lagrange}), this feature will be necessary later in the proof of Lemma \ref{lem Liouville}.\\

Another important obstacle in these kind of arguments is that the perturbation sets $A_\epsilon$ might have a lot of eccentricity (large diameter, but with low density). This means that the change in non-local energy can be several order of magnitudes larger than the available volume bounds for $A_\epsilon$. In fact, $A_\epsilon$ may have a diameter of order $\epsilon$ whereas we only know in general that  $|A_\epsilon|$ is only $\geq \epsilon^{2n}$, thus, one could be concerned that such perturbations do not detect the almost minimality property.\\

In an earlier version of this manuscript an attempt to fix this issue consisted in constructing a different perturbation than in \cite{CRS10}, however such a construction did not solve the problem either, as was pointed out to the authors by the anonymous referee. After this realization it was observed that possibly larger terms arising from the original perturbation argument can all be absorbed into $\rho(.)$, and this ``technicality'' is completely taken care of by Lemma \ref{lem rho doubling}. This is still not entirely satisfactory, for the price we pay is not being  able to prove that the first variation of $J_s$ energy vanishes when $\rho = o(t^s)$ which is to be expected (see remark after Theorem \ref{thm second Euler Lagrange}), however, these estimates are more than enough for the regularity theory.}\\

\textbf{Construction of the perturbation}. Suppose $B=B_{2R}(-2Re_n)$ is a ball contained in $E$ touching $\partial E$ at $0$, assuming always $R\geq 1$. For any \BlueComment{$\epsilon\in(0,R)$}, we will denote by $T_\epsilon(x)$ an involution which should be thought of as a reflection with respect to a perturbation of the sphere  $\partial B_{R+\epsilon}(-Re_n)$.\\

\BlueComment{More precisely, $T_\epsilon(x)$ is the unique point $x_1$ lying on the ray from $-Re_n$ to $x$ such that the mid point of the segment $xx_1$ lies on a slightly deformed sphere $\partial V_{R,\epsilon}$, where
\begin{equation*}
V_{R,\epsilon}=\{x \in \mathbb{R}^n \;|\;|x+Re_n|\leq R+d_\epsilon(x)\},
\end{equation*}
\begin{equation*}
\mbox{ where } d_\epsilon(x)= \epsilon^2 d(x/\epsilon),\; d(x)=R^{-1}(1-|x'|^2)_{_+}.	
\end{equation*}

Algebraically, the map is given by
\begin{equation*}
T_\epsilon(x) +x =  -2Re_n+2(R+d_\epsilon(x))\frac{x+Re_n}{|x+Re_n|}.
\end{equation*}

Observe that $T_\epsilon$ is a smooth diffeomorphism and an involution of the interior of the ``ring'' $V_{2R,\epsilon}\setminus V_{0,\epsilon}$ into itself}. We are going to perturb the set $E$ by adding to it $A_\epsilon$, defined by 
\begin{equation*}
A_\epsilon=A^-_\epsilon \cup T_\epsilon(A^-_\epsilon) \setminus E \;\;\; \mbox{ where } \BlueComment{A^-_\epsilon = V_{R,\epsilon} \setminus E.}
\end{equation*}
An observation that will be important in all what follows is that $A_\epsilon$ can be decomposed as 
\begin{equation*}
A_\epsilon = S_\epsilon \cup D_\epsilon,
\end{equation*}
where $T_\epsilon(S_\epsilon)=S_\epsilon$ and 
\begin{equation*}
D_\epsilon \subset \BlueComment{V_{R,\epsilon}\setminus E.}
\end{equation*}
The other properties of  $A_\epsilon$ and $T_\epsilon$ that we need are summarized in the following proposition.


\BlueComment{\begin{prop}\label{prop Properties} Assume that $B_{2R}(-2Re_n)$ is an interior ball tangent to $\partial E$ at $0$ with $R\geq 1$ and that $T_\epsilon$ is as above. Setting $r_{x,\epsilon}:=d(x,\partial V_{R,\epsilon})$, the following holds $\forall\; \epsilon \in (0,(6n)^{-1})$	
\begin{enumerate}
\item If $x \in V_{2R,\epsilon}\setminus V_{0,\epsilon}$ we have \footnote{ here we are using the Euclidean norm for matrices $|A|=\mbox{Tr}(AA^*)^{1/2}$}
\begin{equation*}
|DT_\epsilon(x)- P_x |\leq \frac{2}{R}(3nr_{x,\epsilon}+|x'|),
\end{equation*}
where $P_x =$ reflection along the direction $x+Re_n$. In particular,
\begin{equation*}
\left | \frac{|T_\epsilon(x)-T_\epsilon(y)|}{|x-y|}-1 \right |\leq \frac{2}{R}\max\{3nr_{x,\epsilon}+|x'|,3nr_{y,\epsilon}+|y'|\},
\end{equation*}
whenever $x, y \in V_{2R,\epsilon}\setminus V_{0,\epsilon}$.

\item We have the inclusions $A^-_\epsilon \subset B_{2\epsilon}$ and $B_{\epsilon^2/(2R)}\setminus E \subset A_\epsilon \subset B_{8\epsilon}$.

\end{enumerate}

\end{prop}}

\begin{proof}
\BlueComment{To show the first assertion we are going to quantify how far is $T_\epsilon$ from being a isometry. First let us differentiate $T_\epsilon$,
\begin{equation*} 
DT_\epsilon(x)=\left (-1+2\frac{(R+d_\epsilon(x))}{|x+Re_n|} \right ) I - 2\frac{(R+d_\epsilon(x))}{|x+Re_n|}\frac{x+Re_n}{|x+Re_n|}\otimes \frac{x+Re_n}{|x+Re_n|} .
\end{equation*}
\begin{equation*}
+\epsilon (\nabla d)(x/\epsilon) \otimes \frac{(x+Re_n)}{|x+Re_n|}.	
\end{equation*}

From this expression we observe already that when $|x+Re_n|=R+d_\epsilon(x)$ (i.e. when $x\in \partial V_{R,\epsilon}$) then $DT_\epsilon$ is nothing but a reflection (specifically, $P_x$) plus a ``defect'' term $\nabla d(x)\otimes \frac{x+Re_n}{|x+Re_n|}$. In general, note that $|x+Re_n|= R+d_\epsilon(x) \pm r_{x,\epsilon}$ and	
\begin{equation*}
DT_\epsilon(x)=P_x+2\left (\frac{(R+\epsilon)}{R+d_\epsilon(x) \pm r_{x,\epsilon}} -1\right ) \left (I-2 \frac{x+Re_n}{|x+Re_n|}\otimes \frac{x+Re_n}{|x+Re_n|} \right )
\end{equation*}
\begin{equation*}
+\epsilon(\nabla d)(x/\epsilon) \otimes \frac{(x+Re_n)}{|x+Re_n|}.
\end{equation*}
Since $\left |\frac{(R+\epsilon)}{R+d_\epsilon(x) \pm r_{x,\epsilon}} -1\right |=\frac{r_{x,\epsilon}}{R+d_\epsilon(x)\pm r_{x,\epsilon}}$, and $|\nabla d(x)|\leq 2R^{-1}|x'|$ we obtain
\begin{equation*}
|DT_\epsilon(x)-P_x|\leq 3n\frac{r_{x,\epsilon}}{R+d_\epsilon(x)- r_{x,\epsilon}}+\frac{2|x'|}{R}\leq \frac{6n}{R}r_{x,\epsilon}+\frac{2|x'|}{R},
\end{equation*}
provided $\epsilon<R/2$. Now, since $P_x$ is an isometry, by integrating along the segment between $x$ and $y$ we get (for all $x,y$ in the ring under consideration)
\begin{equation*}
|T_\epsilon(x)-T_\epsilon(y)|\leq \left ( 1+\frac{2}{R}\max\{3nr_{x,\epsilon}+|x'|,3nr_{y,\epsilon}+|y'|\}\right )|x-y|.
\end{equation*}

This gives the desired upper bounded for $\frac{|T_\epsilon(x)-T_\epsilon(y)|}{|x-y|}$. Given that $T_\epsilon$ is an involution, we get the lower bound too and we are done proving the first assertion.

Now to the second part: by construction,
\begin{equation*}
A^-_\epsilon \subset V_{R,\epsilon} \setminus B_{2R}(-2Re_n),
\end{equation*}
in other words, if $x \in A^-_\epsilon \subset V_{R,\epsilon} \setminus B_{2R}(-2Re_n)$ then
\begin{equation*}
|x+Re_n|\leq R+\epsilon d(x'/\epsilon),\;\;|x+2Re_n|>2R.	
\end{equation*}

From these two inequalities follows that $d(x'/\epsilon) \neq 0 \Rightarrow |x'|<\epsilon$. Moreover by a straightforward computation it can be checked that $x_n \geq -R^{-1}|x'|^2$, and recalling $\epsilon \leq R$ we get $x_n>-\epsilon$, the upper bound $x_n<\epsilon$ is obvious.	All this proves that $|x|\leq 2\epsilon$, that is $A^-_\epsilon \subset B_{2\epsilon}$. Now using the estimate on $T_\epsilon$ we obtain 
\begin{equation*}
|DT_\epsilon|\leq 1+ 2\frac{3n+1}{R}\epsilon \leq 4.
\end{equation*}

Here we used again $\epsilon<(6n)^{-1}$ and $R\geq 1$, this gives a bound for a ball trapping $T_\epsilon(A_\epsilon^-)$. In fact, we have proved that $A_\epsilon	\subset B_{8\epsilon}$. Finally, since $R\geq 1$ we have $|x|\leq \epsilon^2/(2R) \Rightarrow d_\epsilon(x')\geq 3\epsilon^2/(4R)$ which implies $|x+Re_n|\leq R+d_\epsilon(x')$, this proves that $V_{R,\epsilon}$ contains $B_{\epsilon^2/(2R)}$, and thus $B_{\epsilon^2/(2R)} \setminus E \subset A^-_\epsilon$ and we are done with the proof.}
\end{proof}

We are now ready to state and prove the sharpening of Theorem \ref{thm first Euler Lagrange} that we need in our setting.


\BlueComment{

\begin{thm}\label{thm second Euler Lagrange}(Euler-Lagrange inequalities)
Let $E$ be a set satisfying the supersolution property (\ref{eqn supersolution}) in $\Omega$ with respect to $(J_s,\rho,d)$, $\rho$ satisfying A1-A3. Suppose that $0 \in \partial E$ and that $E \cap \Omega$ contains the ball $B_{2R}(-2Re_n), R\geq 1$.\\

There are constants $C_0(n,s,\rho)$ and $r_0(n,s,\rho)$ such that:
\begin{equation*}
\mbox{whenever } 0<8\epsilon^*<\delta<\min\{r_0(n,s,\rho),d\}, 
\end{equation*}
\begin{equation*}
\mbox{ one can find at least one } \epsilon \in (\epsilon^*/2,\epsilon^*) \mbox{ such that:}	
\end{equation*}
\begin{equation}\label{eqn Euler Lagrange}
L(A_\epsilon,E \setminus B_{\delta}) - L(A_\epsilon,E^c \setminus B_{\delta})\leq C_0\left (\rho(8\epsilon)\epsilon^{n-s}+R^{-1}\delta^{\frac{1-s}{2}}|A^-_\epsilon|. \right )
\end{equation}

An analogous inequality holds when $E$ satisfies (\ref{eqn subsolution}) and $B_{2R}(Re_n)\subset E^c \cap \Omega$.

\bigskip

\end{thm}

\begin{proof} 

We divide the proof in 3 steps. In the first step we use directly the definition of the supersolution property (\ref{eqn supersolution}) to show that the left side in (\ref{eqn Euler Lagrange}) is controlled by $\rho(.)\epsilon^{n-s}$ plus a term containing the singularity of the kernel (i.e. the higher order part). In the second step we show there are cancellations in the extra terms due to the symmetry of $S_\epsilon$ under $T_\epsilon$ that help us get rid of the high order singularity, this is where the built-in symmetry of $A_\epsilon$ becomes crucial. Finally, we use again the supersolution property to control whatever is left by $L(A^-_\epsilon,\mathcal{C}B_{R+\epsilon}(-Re_n))$.\\

    
\noindent \emph{Step 1}. Fix $\epsilon,\delta \in (0,d)$ with $8\epsilon<\delta$ so that $A_\epsilon \subset B_\delta$ by Proposition \ref{prop Properties}. Define
\begin{equation*}
F_{\epsilon,\delta}:=T_\epsilon\left ( (E^{c} \setminus A_\epsilon) \cap B_{\delta}\right ).
\end{equation*}

The first step consists in proving the inequality
\begin{equation*}
L(A_{\epsilon},E\setminus B_{\delta})-L(A_{\epsilon},E^{c} \setminus B_{\delta}) \leq \rho(8\epsilon)(8\epsilon)^{n-s}
\end{equation*}
\begin{equation}\label{eqn step one}
-\left [ L(A_{\epsilon},F_{\epsilon,\delta})-L(A_{\epsilon},T(F_{\epsilon,\delta})) \right ].
\end{equation}

To do this, we shall use (\ref{eqn supersolution}) for $A=A_\epsilon$, with $r=8\epsilon$. We decompose the integrals in (\ref{eqn supersolution}) into parts corresponding to integration inside and outside the ball $B_\delta$, as follows
\begin{equation*}
L(A_\epsilon,E\setminus B_{\delta})-L(A_{\epsilon},E^{c} \setminus B_{\delta}) \leq \rho(8\epsilon)(8\epsilon)^{n-s} -I_0,
\end{equation*}
\begin{equation*}
I_0= L(A_{\epsilon},E \cap B_{\delta})-L(A_{\epsilon},( E^{c}\setminus A_{\epsilon} )\cap B_{\delta}).
\end{equation*}

By construction, it is obvious that $F_{\epsilon,\delta} \subset E$. Furthermore, we claim $T_\epsilon$ maps $B_\delta \setminus V_{R,\epsilon}$ into $B_\delta$. To prove this let us take
\begin{equation*}
x \in B_\delta \setminus V_{R,\epsilon}.
\end{equation*}

Then, we observe that the normal hyperplane to $x+Re_n$ passing through $-Re_n+(R+d_\epsilon)\frac{x+Re_n}{|x+Re_n|}$ separates $x$ and the origin, so that the distance between $0$ and $x$  is no smaller than the distance between $0$ and the reflection of $x$ by the aforementioned hyperplane, this image is just $T_\epsilon(x)$. This shows $|T_\epsilon(x)|\leq |x|\leq \delta$ and the claim is proved.  We conclude that  $F_{\epsilon,\delta} \subset E \cap B_\delta$ and
\begin{equation*}
I_0=L(A_{\epsilon},E \cap B_{\delta}\setminus F_{\epsilon,\delta})+\left [ L(A_{\epsilon},F_{\epsilon,\delta})-L(A_{\epsilon},T_\epsilon(F_{\epsilon,\delta})) \right ],
\end{equation*}
given that the first term is non-negative, (\ref{eqn step one}) is proved.

\bigskip


\noindent \emph{Step 2}.  Using $A_\epsilon = S_\epsilon \cup D_\epsilon$, we rewrite the integrals on the right hand side of (\ref{eqn step one}),

\begin{equation*}
L(A_{\epsilon},F_{\epsilon,\delta})-L(A_{\epsilon},T_\epsilon(F_{\epsilon,\delta}))=\left [L(S_\epsilon,F_{\epsilon,\delta})-L(S_{\epsilon},T_\epsilon(F_{\epsilon,\delta})) \right ]
\end{equation*}
\begin{equation*}
+ \left [L(D_{\epsilon},F_{\epsilon,\delta})-L(D_{\epsilon},T_\epsilon(F_{\epsilon,\delta})) \right ].
\end{equation*}

Let us bound each term. We apply the change of variables $x,y \to T_\epsilon(x),T_\epsilon(y)$ to estimate the integral on $S_\epsilon \times T_{\epsilon}(F_{\epsilon,\delta})$ as follows

\begin{equation*}
L(S_{\epsilon},T_\epsilon(F_{\epsilon,\delta})) =(1-s)\int_{T_\epsilon(F_{\epsilon,\delta})}\int_{S_\epsilon}\frac{1}{|x-y|^{n+s}}dxdy
\end{equation*}
\begin{equation*}
=(1-s)\int_{F_{\epsilon,\delta}}\int_{S_\epsilon}\frac{\mbox{det}\left ( DT_\epsilon(x) \right ) \mbox{det} \left (DT_\epsilon(y)\right )}{|T_\epsilon(x)-T_\epsilon(y)|^{n+s}}dxdy.
\end{equation*}

The first part of Proposition \ref{prop Properties} asserts that $|DT_\epsilon(x)-P_x| \leq 2R^{-1}\left (3nr_{x,\epsilon}+|x'|\right)$ for all $x$ under consideration, which gives an upper bound\footnote{under our assumptions $|DT_\epsilon(x)-I|\leq 2$, thus $|\mbox{det}(DT_\epsilon)| \leq (1+|DT_\epsilon(x)-P_x|)^n\leq 1+n2^n|DT_\epsilon(x)-P_x|$} for the Jacobians, leading to
$$\begin{array}{l}
L(S_{\epsilon},T_\epsilon(F_{\epsilon,\delta}))\leq (1-s)\int_{F_{\epsilon,\delta}}\int_{S_\epsilon}\frac{1+C_nR^{-1}\max\{3nr_{x,\epsilon}+2|x'|,3nr_{y,\epsilon}+2|y'|\}}{|T_\epsilon(x)-T_\epsilon(y)|^{n+s}}dxdy.	
\end{array}$$

Also by Proposition \ref{prop Properties} we may bound $|T_\epsilon(x)-T_\epsilon(y)|^{-n-s}$ from above with $|x-y|^{-n-s}$, times a factor given by Proposition \ref{prop Properties}, namely
\begin{equation*}
|T_\epsilon(x)-T_\epsilon(y)|^{-n-s}\leq [1-2R^{-1}\max\{3nr_{x,\epsilon}+2|x'|,3nr_{y,\epsilon}+2|y'|\}]^{-n-s}|x-y|^{-n-s}	
\end{equation*}

Putting these two bounds together we conclude that
\begin{equation*}
L(S_{\epsilon},T_\epsilon(F_{\epsilon,\delta})) \leq L(S_\epsilon,F_{\epsilon,\delta})
\end{equation*}
\begin{equation*}
+(1-s)\int_{F_{\epsilon,\delta}}\int_{S_\epsilon}\frac{1+C_nR^{-1}\max\{3nr_{x,\epsilon}+2|x'|,3nr_{y,\epsilon}+2|y'|\}}{|x-y|^{n+s}}dxdy.	
\end{equation*}

For the second term we also do a change of variables but only in $y$, using again Proposition \ref{prop Properties} to bound the Jacobian. However, we control the denominator in a different manner, arguing as in the proof of the claim in step 1 to conclude $|x-T_\epsilon(y)|\geq |x-y|$ for all $x,y \in V_{R,\epsilon}$ (remember $x \in D_\epsilon,y \in F_{\epsilon,\delta} \Rightarrow x,y \in V_{R,\epsilon}$). In this case,
\begin{equation*}
L(D_{\epsilon},T_\epsilon(F_{\epsilon,\delta})) =(1-s)\int_{T_\epsilon(F_{\epsilon,\delta})}\int_{D_\epsilon}\frac{1}{|x-y|^{n+s}}dxdy \leq
\end{equation*}
\begin{equation*}
\leq(1-s)\int_{F_{\epsilon,\delta}}\int_{D_\epsilon}\frac{1+C_nR^{-1}\max\{3nr_{x,\epsilon}+2|x'|,3nr_{y,\epsilon}+2|y'|\}}{|x-y|^{n+s}}dxdy.
\end{equation*}

The bounds we have obtained for $L(S_\epsilon,T_\epsilon(F_{\epsilon,\delta}))$ and $L(D_\epsilon,T_\epsilon(F_{\epsilon,\delta}))$ give us,
\begin{equation*}
\begin{array}{l}
L(A_{\epsilon},F_{\epsilon,\delta})-L(A_{\epsilon},T_\epsilon(F_{\epsilon,\delta})) \geq \\
-(1-s)\int_{F_{\epsilon,\delta}}\int_{A_\epsilon}\frac{C_nR^{-1}\max\{3nr_{x,\epsilon}+2|x'|,3nr_{y,\epsilon}+2|y'|\}}{|x-y|^{n+s}}dxdy.
\end{array}
\end{equation*}

Now, denote by $A^+_\epsilon$ the set $A_\epsilon \setminus A^-_\epsilon$, observe that $T_\epsilon(A^+_\epsilon)\subset A^-_\epsilon$. It is reasonable to try to control $L(A^+_\epsilon,F_{\epsilon,\delta})$ in terms of $L(A^-_\epsilon,F_{\epsilon,\delta})$, that is, using the change of variables $x \to T_\epsilon(x)$,
\begin{equation*}
L(A^+_\epsilon,F_{\epsilon,\delta}) = (1-s)\int_{F_{\epsilon,\delta}}\int_{A^+_\epsilon}\frac{1}{|x-y|^{n+s}}dxdy
\end{equation*}
\begin{equation*}
= (1-s)\int_{F_{\epsilon,\delta}}\int_{T_\epsilon(A^+_\epsilon)}\frac{\mbox{det}(DT_\epsilon(x))}{|T_\epsilon(x)-y|^{n+s}}dxdy
\end{equation*}
\begin{equation*}
\leq (1-s)\int_{F_{\epsilon,\delta}}\int_{A^-_\epsilon}\frac{1+C_nR^{-1}\max\{3nr_{x,\epsilon}+2|x'|,3nr_{y,\epsilon}+2|y'|\}}{|x-y|^{n+s}}dxdy,
\end{equation*}
and we  get 
\begin{equation*}
\begin{array}{l}
L(A_{\epsilon},F_{\epsilon,\delta})-L(A_{\epsilon},T_\epsilon(F_{\epsilon,\delta})) \geq\\ -(1-s)\int_{F_{\epsilon,\delta}}\int_{A^-_\epsilon}\frac{C_nR^{-1}\max\{3nr_{x,\epsilon}+2|x'|,3nr_{y,\epsilon}+2|y'|\}}{|x-y|^{n+s}}dxdy. 
\end{array}\end{equation*}
	
Note that the quantity in the numerator is of order at most $\delta$ (since the integration is over subsets of $B_\delta$), which is not sharp enough for our purposes. Nevertheless, when $x$ and $y$ are far apart the kernel is bounded so in that region we can get a better bound. Specifically, we have
\begin{equation*}
\max\{3nr_{x,\epsilon}+2|x'|,3nr_{y,\epsilon}+2|y'|\}\leq \left \{\begin{array}{rr}
6n|x-y|+4\epsilon & \mbox{ if } |x-y|\geq 2r_{x,\epsilon}\\
14n\epsilon & \mbox{ if } |x-y|< 2r_{x,\epsilon}
\end{array}\right. .
\end{equation*}

This follows at once from the inequalities
\begin{equation*}
r_{y,\epsilon}\leq r_{x,\epsilon}+|x-y|,\;\;|y'|\leq |x'|+|x-y|
\end{equation*}
\begin{equation*}
\Rightarrow 3nr_{y,\epsilon}+2|y'|\leq 	3nr_{x,\epsilon}+2|x'|+4n|x-y|,
\end{equation*}

Together with the fact that $r_{x,\epsilon}\leq \epsilon$ and $|x| \leq 2\epsilon$ for all $x\in A^-_\epsilon$, which was shown in Proposition \ref{prop Properties}. Then it follows that
\begin{equation*}
\begin{array}{l}
\int_{F_{\epsilon,\delta}}\frac{\max\{3nr_{x,\epsilon}+2|x'|,3nr_{y,\epsilon}+2|y'|\}}{|x-y|^{n+s}}dy\\
\leq C_n\left (\int_{F_{\epsilon,\delta} \setminus B_{2r_{x,\epsilon}}(x)}\frac{|x-y|+\epsilon}{|x-y|^{n+s}}dy+\int_{F_{\epsilon,\delta} \cap B_{2r_{x,\epsilon}}(x)}\frac{\epsilon}{|x-y|^{n+s}}dy	\right ).
\end{array}\end{equation*}

Integrating in $x\in A^-_\epsilon$ (and noting that $|x-y|^{-n-s+1}$ is locally integrable), this gives
\begin{equation*}
(1-s)\int_{A^-_\epsilon}\int_{F_{\epsilon,\delta}}\frac{\max\{3nr_{x,\epsilon}+2|x'|,3nr_{y,\epsilon}+2|y'|\}}{|x-y|^{n+s}}dydx	
\end{equation*}
\begin{equation*}
\leq C_n\left (\delta^{1-s}|A^-_\epsilon|+ \epsilon	L(A^-_\epsilon,F_{\epsilon,\delta})\right ),	
\end{equation*}
and using our previous bound for $I_0$ we get
\begin{equation}\label{eqn step two}
-I_0 \leq C_nR^{-1}\left (|A^-_\epsilon|\delta^{1-s}+\epsilon L(A^-_\epsilon,F_{\epsilon,\delta})\right ).
\end{equation}

This finishes the second step, next we show that $L(A^-_\epsilon,F_{\epsilon,\delta})$ is not too big.
\bigskip


\emph{Step 3}. Recall from Step 1 that $F_{\epsilon,\delta}\subset E$. Therefore $L(A^-_\epsilon,F_{\epsilon,\delta})\leq L(A^-_\epsilon,E)$, and this last quantity (by the supersolution assumption (\ref{eqn supersolution})) is not bigger than
\begin{equation*}
L(A^-_\epsilon,E^c\setminus A^-_\epsilon)+\rho(8\epsilon)(8\epsilon)^{n-s}.
\end{equation*}

Moreover, since $E^c \setminus A^-_\epsilon\subset V_{R,\epsilon}^c$ we conclude
\begin{equation*}
L(A^-_\epsilon,F_{\epsilon,\delta}) \leq L(A^-_\epsilon,V_{R,\epsilon}^c)+	\rho(8\epsilon)(8\epsilon)^{n-s}.
\end{equation*}

\emph{Claim \footnote{Compare with Lemma 5.2 in \cite{CRS10}.}}: There exists $\epsilon_0(n,s,\rho) \in (0,1)$ such that for any $\epsilon^*<\epsilon_0$ there exists some $\epsilon\in (\epsilon^*,2\epsilon^*)$ such that
\begin{equation*}
L(A^-_\epsilon,V_{R,\epsilon}^c)\leq C_{n,s}\epsilon^{-\frac{s+1}{2}}|A^-_\epsilon|.
\end{equation*}

Again following \cite{CRS10}, we (essentially) rewrite $L(A^-_\epsilon,V_{R,\epsilon}^c)$ in terms of
\begin{equation*}
a(r):=\mathcal{H}^{n-1}\left (A^-_\epsilon \cap \partial V_{R,r} \right ), \;\;0<r<\epsilon.
\end{equation*}

Indeed, if $x \in A^-_\epsilon\subset V_{R,\epsilon}$ then for $d(x):=d(x,\partial V_{R,\epsilon})$
\begin{equation*}
\int_{V^c_{R,\epsilon}}\frac{1}{|x-y|^{n+s}}dy\leq n|B_1|\int_{d(x)}^\infty\frac{1}{r^{n+s}}r^{n-1}dr	
\end{equation*} 
\begin{equation*}
\Rightarrow \int_{V^c_{R,\epsilon}}\frac{1}{|x-y|^{n+s}}dy\leq \frac{n|B_1|}{s\; d(x)^s}.
\end{equation*}

Thanks to this, we get the bound
\begin{equation*}
L(A^-_\epsilon,V^c_{R,\epsilon})\leq C_{n,s}\int_{A^-_\epsilon}\frac{1}{d(x,\partial V_{R,\epsilon})^s}dx,	
\end{equation*}
since $d(A^-_\epsilon,\partial V_{R,\epsilon})\leq \epsilon$ and $d(x,\partial V_{R,\epsilon}) \sim (\epsilon-r)$ if $x \in \partial V_{R,r}$ ($r<\epsilon$) this implies that
\begin{equation*}
L(A^-_\epsilon,V^c_{R,\epsilon})\leq C_{n,s}\int_0^\epsilon \frac{a(r)}{(\epsilon-r)^{s}}dr.	
\end{equation*}

Then, let us suppose that the \emph{Claim} is not true, then there exists a small $\epsilon^*$ such that the following inequality holds for every $\epsilon \in (\epsilon^*,2\epsilon^*)$ (for brevity we write $\eta= \frac{s+1}{2}$)
\begin{equation*}
\int_0^\epsilon \frac{a(r)}{(\epsilon-r)^{s}}dr> \epsilon^{-\eta}\int_0^\epsilon a(r) dr.	
\end{equation*}

We integrate this for $\epsilon \in (\epsilon^*,2\epsilon^*)$ and change the order of integration to get for $f(r)=|V_{R,r} \setminus E|$ defined for $r \in (0,\epsilon)$ (recall that $f'(r)=a(r)$ by the coarea formula)
\begin{equation*}
\int_{\epsilon^*}^{2\epsilon^*}f'(r)(2\epsilon^*-r)^{1-s}dr+\int_0^{\epsilon^*}f'(r)\left ( (2\alpha-r)^{1-s}-(\alpha-r)^{1-s}\right )dr 
\end{equation*}
\begin{equation*}
>\frac{1-s}{1-\eta}\left [ \int_{\epsilon^*}^{2\epsilon^*}f'(r)\left [(2\epsilon^*)^{1-\eta}-r^{1-\eta}\right ]dr+ \right .
\end{equation*}
\begin{equation*}
\left. +\int_0^{\epsilon^*}f'(r)\left ( (2\alpha)^{1-\eta}-\alpha^{1-\eta}\right )dr \right ].	
\end{equation*}

Observe that the second term on the left hand side is bounded from above by $2^{1-s}f(\epsilon^*)(\epsilon^{*})^{1-s}$ and the second term on the right hand side of the inequality is bounded from below by $C_{s}f(\epsilon^*)(\epsilon^*)^{1-\eta}$, this shows that if $\epsilon^*$ is smaller than some $\epsilon_0=\epsilon_0(s)$ then
\begin{equation*}
\int_{\epsilon^*}^{2\epsilon^*}f'(r)(2\epsilon^*-r)^{1-s}dr\geq C_{s}\int_{\epsilon^*}^{2\epsilon^*}f'(r)\left [(2\epsilon^*)^{1-\eta}-r^{1-\eta}\right ]dr
\end{equation*} 
\begin{equation*}
\Rightarrow (\epsilon^*)^{1-s}\int_{\epsilon^*}^{2\epsilon^*}f'(r)dr\geq C_{s,\eta} (\epsilon^*)^{1-\eta}\int_{\epsilon^*}^{2\epsilon^*}f'(r)dr.
\end{equation*} 

Observe that we can use the same proof of Proposition \ref{prop density} to get density bounds for $E^c$ with respect to the modified balls $V_{R,\epsilon}$. In particular $f(r)$ satisfies an integro-differential inequality like (\ref{eqn density ode}), from where it follows easily that $\int_{\epsilon^*}^{2\epsilon^*}f'(r)dr\neq 0$ for all positive and small enough $\epsilon^*$. This means we can cancel the common factors and get $(\epsilon^*)^{\eta-s}\geq C_{s}$. Recalling $\eta=\frac{1+s}{2}>s$ we get a contradiction if $\epsilon^*$ is too small and the claim is proved.\\ 

Using the bound for $L(A^-_\epsilon,F_{\epsilon,\delta})$ together with step 2 we obtain the bound
\begin{equation*}
-I_0\leq C_nR^{-1}\left (\delta^{1-s}|A^-_\epsilon|+\epsilon^{\frac{1-s}{2}}|A^-_\epsilon|+ \epsilon\rho(8\epsilon)\epsilon^{n-s}\right ),
\end{equation*}
for some $\epsilon \in (\epsilon^*,2\epsilon^*)$ as long as $\epsilon^*<\epsilon_0$, and as long as $\delta<r_0$ for some $r_0=r_0(n,s,\rho)$. Moreover, since $\epsilon<\delta$ we will for the sake of brevity bound the term $\epsilon^{\frac{1-s}{2}}|A^-_\epsilon|$ by $\delta^{\frac{1-s}{2}}|A^-_\epsilon|$. Putting all this together with the bound (\ref{eqn step one}) from step 1 the theorem is finally proved.
\end{proof}
}
As a corollary of this proof, we get a pointwise Euler-Lagrange inequality whenever \BlueComment{$\rho(t)\leq t^{p}$ ($p>n+s$). This is far from a sharp estimate, as it should hold for any $p>s$, however, this requires a more refined perturbation argument than that of Theorem \ref{thm second Euler Lagrange}. The main justification in presenting this corollary is that it will introduce an estimate used later in the proof \footnote{a posteriori one might use the knowledge that $\partial E$ is smooth to prove this for all $p>s$ using for instance perturbations by smooth vector fields.}.}

\begin{cor}\label{cor pointwise Euler Lagrange}
Suppose $E$ satisfies the supersolution property  (\ref{eqn supersolution}) in  $\Omega$ with respect to a $\rho$ such that \BlueComment{$\rho(t)\leq Ct^p, p>n+s$}. Then anytime $E$ has an interior tangent ball at a point $x_0 \in \Omega$ we have

\begin{equation*}
\int_{\mathbb{R}^n}\frac{\chi_E(x)-\chi_{E^c}(x)}{|x-x_0|^{n+s}}dx \leq \BlueComment{0}.
\end{equation*}

As in the \BlueComment{previous } theorem, we get the opposite inequality in the case of a subsolution and an exterior tangent ball.
\end{cor}

\textbf{Remark.} Given $x_0 \in \partial E$, the quantity

\begin{equation*}
-(1-s)\int_{\mathbb{R}^n}\frac{\chi_E(x)-\chi_{E^c}(x)}{|x-x_0|^{n+s}}dx.
\end{equation*}

is what is known as the \textbf{non-local mean curvature} of $\partial E$ at $x_0$. \BlueComment{It has many features similar to the usual mean curvature, } for instance, thanks to the invariance of the kernel $|x-y|^{-n-s}$ under rotations, it is easy to see that a ball of radius $r$ has constant non-local mean curvature and that this constant is proportional to $r^{-s}$. Further, bounded convex sets $E$ would have boundaries with positive non-local mean curvature.

\begin{proof}
\BlueComment{Without loss of generality, we may take $x_0=0$. Let $\epsilon$ and $\delta$ satisfy the hypothesis of Theorem \ref{thm second Euler Lagrange}, observe that
\begin{equation*}
\left | \frac{L(A_\epsilon,E\setminus B_{\delta})-L(A_{\epsilon},E^{c} \setminus B_{\delta})}{|A_\epsilon|}-\int_{B_\delta^c}\frac{\chi_E(x)-\chi_{E^c}(x)}{|x|^{n+s}}dx\right |
\end{equation*}
\begin{equation*}
\leq \int_{B_\delta^c}\frac{C_n\epsilon}{|x|^{n+s}}dx \leq C_n\epsilon\delta^{-1-s}.
\end{equation*}

This follows from the Lipschitz bound for $|x-y|^{-n-s}$ in each of the sets $\{|x-y|\geq \delta\}$ and the fact that $A_\epsilon \subset B_{8\epsilon }$. Then Theorem \ref{thm second Euler Lagrange} implies that
\begin{eqnarray*}
\int_{B_\delta^c}\frac{\chi_E(x)-\chi_{E^c}(x)}{|x|^{n+s}}dx \leq C_n\epsilon\delta^{-1-s}+C_0\frac{1}{R}\left (f(8\epsilon)\epsilon^{n-s}|A^-_\epsilon|^{-1}+\delta^{\frac{1-s}{2}} \right ).
\end{eqnarray*}

Recall that by Propositions \ref{prop density} and \ref{prop Properties} we have $|A^-_\epsilon|\geq C_0\epsilon^{2n}$, therefore
\begin{eqnarray*}
\int_{B_\delta^c}\frac{\chi_E(x)-\chi_{E^c}(x)}{|x|^{n+s}}dx \leq C_n\epsilon\delta^{-1-s}+C_0\frac{1}{R}\left (C_0f(8\epsilon)\epsilon^{-n-s}+\delta^{\frac{1-s}{2}} \right ),
\end{eqnarray*}
since $f(8\epsilon)\leq C_n\epsilon^{p},\;p>n+s$ we are free to let $\epsilon\to 0$ with $\delta>0$ fixed, getting
\begin{equation*}
\int_{B_\delta^c}\frac{\chi_E(x)-\chi_{E^c}(x)}{|x|^{n+s}}dx \leq C_0\frac{1}{R}\delta^{1-s},
\end{equation*}
taking the limit as $\delta\to 0$ the corollary is proved.}
\end{proof}

To finish the section, we prove the Harnack estimate. This is the key tool needed for the improvement of flatness argument. We will use the following definition \BlueComment{(see also discussion at the beginning of Section 6)}:

\begin{DEF}
A set $E \subset \mathbb{R}^n$ is said to be flat of order $a$ at $x_0$ with respect to \BlueComment{$\gamma:[0,1] \to \mathbb{R}_+$,  if $a=\gamma(2^{-k})$ for some $k \in \mathbb{N}$ } and there is a sequence of unit vectors $\{e_l\}_{\BlueComment{l}}$ such that
\begin{equation*}
\partial E \cap B_{2^l}(x_0) \subset \{ |(x-x_0) \cdot e_l| \leq 2^l\gamma(2^{l-k})\}, \;\; \forall \; \BlueComment{l: \; 0 \leq l \leq k}.
\end{equation*}
\BlueComment{Remark. As we are assuming that $\rho$ is defined in $(0,\delta)$ with $\delta>1$ (see remark before Definition 2.5) we may take the auxiliary function $\hat\rho$ (again see Definition 2.5) as our $\gamma$ above.}

\end{DEF}

For the purposes of the next section (namely showing regularity) it will also be convenient to state the Harnack estimate in a different scale, i.e. scaling the ball of radius $B_{2^{-k}}(0)$ into the ball $B_1(0)$, if a set $E$ is $(J_s,\rho,d)$-minimal, then the rescaled set will be $(J_s.\rho_k,d_k)$-minimal, for $\rho_k(t)=\rho(2^{-k}t)$ and $d_k=\BlueComment{2^{k}}d$.	


\begin{thm}\label{thm Harnack estimate}(Partial Harnack)
For any $\rho$ satisfying assumptions A1-A3, $d>0$ and $s\in(0,1)$ there are positive constants $k_0(n,s,\rho,d)$ and $\delta_0(n,s,\rho,d)$ with the following property:\\

Suppose that \BlueComment{$E$ is $(J_s,\rho_k,d_k$)-minimal } in $B_{2^k}(0)$, flat of order $a=\hat\rho(2^{-k})$ with respect to $\hat\rho$ and that $0 \in \partial E$. Then, if $k>k_0$ the following two inclusions hold:
\begin{equation*}
E \cap B_{\delta_0} \subset \{ x \cdot e_0 <a(1-\delta_0^2)\},\;\;B_{\delta_0} \cap \{ x \cdot e_0 <-a(1-\delta_0^2)\} \subset E.
\end{equation*}

Here $\rho_k(t)=\rho(2^{-k}t), d_k:=2^{-k}d$ and $e_0$ \BlueComment{comes from the flatness hypothesis.}
\end{thm}

\BlueComment{Remark. This theorem says that conditioned to $\partial E$ being ``flat enough'' near zero one has a Harnack inequality. Note that unlike the usual Harnack inequality for regular elliptic equations, we cannot reapply it over and over at finer scales (since the flatness condition might fail in a smaller ball), thus the term ``Partial Harnack''. We get back to this point in Section 6.}

\begin{proof}[Proof of Theorem \ref{thm Harnack estimate}]

Suppose that for given $k$, $a=\hat\rho(2^{-k})$ and $\delta_0$ there is a set $E$ satisfying the assumptions and such that one of the two inclusions does not hold. We may assume that $e_0$ is the positive direction on the $x_n$-axis. With this setup, we shall get a contradiction by showing that if $a$ and $\delta$ are picked universally small then $E$ will ``stretch'' too much, \BlueComment{forcing it to contradict the Euler-Lagrange inequalities (Theorem \ref{thm second Euler Lagrange}).}\\

Without loss of generality (the other case is dealt with similarly), suppose that
\begin{equation}\label{eqn harnack estimate two}
B_{\delta} \cap \{ x_n<-a(1-\delta^2)\} \not\subset E.
\end{equation}

Slide from below a ball $B$ of radius $R \BlueComment{ = } a^{-1} \geq 1$ until it becomes tangent to $\partial E$ from the interior\BlueComment{, which we can do thanks to the flatness assumption. Since we are assuming (\ref{eqn harnack estimate two}), without loss of generality the ball must be tangent at least at one point $z$ such that}
\begin{equation*}
z \in B_{\delta} \cap \{x_n=-a(1-\delta^2) \}.
\end{equation*}

We are under the assumptions of Theorem \ref{thm second Euler Lagrange}, \BlueComment{ thus using the same $\delta$ as in (\ref{eqn harnack estimate two})  and any $\epsilon$ so that $\epsilon$ and $\delta$ are allowed, we have
\begin{equation*}
L(A_\epsilon,E \setminus B_{\delta}(z)) - L(A_\epsilon,E^c \setminus B_{\delta}(z)) \leq C_0\left ( \rho(2^{-k}8\epsilon)\epsilon^{n-s}+a\delta^{\frac{1-s}{2}}|A^-_\epsilon|\right ), 
\end{equation*}
our goal is to obtain an opposite bound. Note that $B \subset E$ so for any $y \in \mathbb{R}^n$ we have}
\begin{equation*}
\int_{B_{\frac{1}{2}} \BlueComment{\setminus B_\delta(z)}}\frac{\chi_E(x)-\chi_{E^c \setminus A_\epsilon}(x)}{|x-y|^{n+s}}dx \geq \int_{B_{\frac{1}{2}}\BlueComment{\setminus B_\delta(z)}}\frac{\chi_B(x)-\chi_{B^c \setminus A_\epsilon}(x)}{|x-y|^{n+s}}dx+\int_{B_{\frac{1}{2}}\BlueComment{\setminus B_\delta(z)}}\frac{\chi_{E \setminus B}(x)}{|x-y|^{n+s}}dx.
\end{equation*}

\BlueComment{Since $0 \in \partial E$ and $z \in B_\delta \cap \{ x_n < -a(1-\delta^2)\}$, the ball $B_{a+2\delta}(y)$ contains $E\cap B_a$ any time $|z-y|<\delta$. This means that the function $|x-y|^{-(n+s)}$ must be greater than $c_n(\delta+a)^{-(n+s)}$ in $E \cap B_{\frac{a}{2}}$, which is a subset of $(E\setminus B) \cap (B_{\frac{1}{2}} \setminus B_\delta(z))$ if $a$ and $\delta$ are small enough and $a\sim \delta$ (recall $z$ lies on $\{x_n = -a(1-\delta^2) \}$ so $B_\delta(z)$ and $B$ are away from the origin). Thanks to Proposition \ref{prop density} this set has measure no less than $c_0a^n$ when $a$ is small. Hence we have
\begin{equation*}
\int_{B_{\frac{1}{2}} \setminus B_\delta(z)}\frac{\chi_{E \setminus B}(x)}{|x-y|^{n+s}}dx \geq c_0(\delta+a)^{-(n+s)}a^{n},\;\; \forall \; y \in B_{\frac{\delta}{2}}(z).
\end{equation*}
}

Additionally, we can check by direct computation that
\BlueComment{
\begin{equation*}
\int_{B_{\frac{1}{2}} \setminus B_\delta(z)}\frac{\chi_B(x)-\chi_{B^c}(x)}{|x-y|^{n+s}}dx \geq -C_n\delta^{-s} \;\; \forall \; y \in B_{\frac{\delta}{2}}(z).
\end{equation*}

\BlueComment{Then, since $\chi_E(x)-\chi_{E^c \setminus A_\epsilon}(x)\geq \chi_B-\chi_{B^c}$  we conclude } there is a universal $C_0$ such that whenever $C_0^{-1}a \leq \delta \leq C_0a$ we have the inequality
\begin{equation}\label{eqn harnack estimate three}
\int_{B_{\frac{1}{2}} \setminus B_\delta(z)}\frac{\chi_E(x)-\chi_{E^c}(x)}{|x-y|^{n+s}}dx \geq \BlueComment{C_0^{-1}\delta^{-s}} \;\;\; \forall \; y \in B_{\frac{\delta}{2}}(z).
\end{equation}}

Next, we estimate this same integral outside $B_{\frac{1}{2}}(0)$. Let $|y| <1/4$ and $r>0$ with $2^{l-1}\leq r < 2^l$  for some $l$ such that $0\leq l \leq k$. Since $E$ is flat of order $a$ with respect to $\hat\rho$, it is not hard to prove that
\begin{equation*}
\left | \int_{|x-y|=r}\frac{\chi_E(x)-\chi_{E^c}(x)}{|x-y|^{n+s}}dS \right | \leq \frac{aC(n,\rho)}{r^{1+\eta_0}},
\end{equation*}
\BlueComment{for some $\eta_0=\eta_0(n,s,\rho)$, this is in part due to assumption A3}. We integrate this with respect to $r$,
\begin{equation*}
\left | \int_{B_{\frac{1}{2}}(0)^c}\frac{\chi_E(x)-\chi_{E^c}(x)}{|x-y|^{n+s}}dx \right |\leq \int_{2^{-1}}^{2^{k}}\frac{aC(n,\rho)}{r^{1+\eta_0}}dr + C_n\int_{2^{k}}^\infty\frac{1}{r^{1+s}}dr
\end{equation*}
\begin{equation*}
\leq \left ( \int_{2^{-1}}^\infty C_0r^{-1-\eta_0}dr\right ) a + \frac{C_n}{s}2^{-sk}, \; \forall\; y \in B_{\frac{1}{4}}.
\end{equation*}

Since $k\geq k_0$, choosing $k_0$ large enough (equivalently, choosing $a$ small enough) to get $2^{-sk} \leq a\ =\hat\rho(2^{-k_0})$ we obtain
\begin{equation}\label{eqn harnack estimate four}
\left | \int_{B_{\frac{1}{2}}(0)^c}\frac{\chi_E(x)-\chi_{E^c}(x)}{|x-y|^{n+s}}dx \right | \leq C_0a, \; \forall y\; \in B_{\frac{1}{4}}.
\end{equation}

The estimates (\ref{eqn harnack estimate three}) and (\ref{eqn harnack estimate four}) put together give us (as long as $\delta<<1$)
\begin{equation*}
\int_{B_{\delta}(z)^c}\frac{\chi_E(x)-\chi_{E^c}(x)}{|x-y|^{n+s}}dx \geq \delta^{-s}-\BlueComment{C_0(n,s,\rho)}a,\;\forall\; y \in B_{\frac{\delta}{2}}(z).
\end{equation*}

From the way we picked $\epsilon$ we \BlueComment{can make sure } that $A_\epsilon \subset B_\frac{\delta}{2}$, this means that we can integrate the above inequality for $y \in A_\epsilon$ and reach the lower bound:
\begin{equation*}
L(A_\epsilon,E \setminus B_{\delta}(z_\epsilon)) - L(A_\epsilon,E^c \setminus B_{\delta}(z_\epsilon)) \geq \delta^{-s}|A_\epsilon|-C(n,s,\rho)a|A_\epsilon|.
\end{equation*}
\smallskip

It is worth emphasizing that this last estimate is independent of the Euler Lagrange inequalities, it follows \emph{only} from the flatness hypothesis\BlueComment{, the density estimate from Proposition \ref{prop density} and the assumption that (\ref{eqn harnack estimate two}) does \emph{not} hold}. Putting these two bounds together we see that for all $\epsilon$, $\delta$ and $a$ under consideration we have
\begin{equation*}
\BlueComment{\delta^{-s}|A_\epsilon|-C(n,s,\rho)a|A_\epsilon| \leq C_0\left (\rho(2^{-k}8\epsilon)\epsilon^{n-s} +\delta^{1+\frac{1-s}{2} }|A^-_\epsilon|\right )}.
\end{equation*}

\BlueComment{We now apply Lemma \ref{lem rho doubling}, which says that if $8\epsilon<1$ then
\begin{equation*}
\rho(2^{-k}8\epsilon)\leq \hat\rho(2^{-k})^m=a^m,	
\end{equation*}
so if $a, \epsilon$ and $\delta$ are universally small ($\epsilon$ being admissible for Theorem \ref{thm second Euler Lagrange}),
\begin{equation*}
\delta^{-s} \leq C_0\left [ a^m\epsilon^{n-s}|A_\epsilon|^{-1}+1\right ].	
\end{equation*}

To finish, recall that $|A_\epsilon|\geq c_0 \epsilon^{2n}$ (Propositions \ref{prop density} and \ref{prop Properties}) and note we can finally pick all parameters so that $\epsilon \sim \delta \sim a$, we get
\begin{equation*}
\delta^{-s}\leq C_0 \left [\delta^{m-n-s}+1 \right ],
\end{equation*}
since $m>n+s$, this means that
\begin{equation*}
\delta^{-s}\leq C_0.
\end{equation*}
}

Choosing $\delta$ (universally) small we get a contradiction. This implies the estimate since we have shown that the first inclusion (\ref{eqn harnack estimate two}) must hold with some (universal) $\delta$ when $k_0$ is (universally) large, since $k_0$ grows as $a \to 0$. The other inclusion is dealt with in the same way, except we must slide a ball $B$ from \emph{above} and use the subsolution version of the Euler-Lagrange inequality. \end{proof}
\section{Flat boundaries are $C^1$}\label{sec flat smooth}

The concept of improvement of flatness is nowadays well understood, it arises not only in the regularity theory of minimal surfaces but on free boundary problems as well. To achieve it here, we make use of a method developed by Savin to address the regularity of level sets in phase transitions \cite{S09}. See also \cite{S07} where a similar idea is applied to elliptic equations.\\ 

Heuristically it goes as follows: by a standard argument, if flat boundaries are not $C^{1}$ then there is a sequence of $(J_s,\rho,d)$ minimal sets $E_k$, with vanishing flatness in $B_{1/2^k}$, and such that $\partial E_k$ cannot be trapped inside a flatter cylinder in $B_{1/2^{k+1}}$.  For this sequence there is a partial Harnack estimate (in our case Theorem \ref{thm Harnack estimate}) which, as we will see below, implies that if we dilate each surface in the direction in which they become flat, then a subsequence of them converges uniformly to the graph of a continuous function u with controlled growth at infinity. This function is then shown to solve the equation $(-\Delta)^{\frac{1+s}{2}}u=0$ in all of $\mathbb{R}^n$ which shows it must be linear. This must give a contradiction, since the uniform convergence and smoothness of the limit force the sequence to lie eventually inside flatter and flatter cylinders, against our initial assumption. Therefore the original sequence cannot exist.\\

There are new technical problems in carrying out this procedure for non-local minimizers, and they were dealt with in \cite{CRS10}. The main difficulty comes from the influence of far away terms in all the estimates. This would not be a problem if we knew that the surface was flat enough away from a neighborhood of $0$, but after scaling a much smaller neighborhood of $0$ into one of size of order 1 we are sending new parts of the surface far away and the non-local terms could become relevant again. In \cite{CRS10} it is shown this is not really a problem, by taking into account not only a small neighborhood of $0$ but also larger and larger balls where the flatness grows until it becomes of order one. This is actually the same as the standard flatness hypothesis, except we are keeping track of the information between the scale of order 1 and the smaller scale. Later in the argument, as we rescale, this information allows us to control the non-local terms. This idea was already used in the statement and proof of the Partial Harnack estimate (see also Definition 5.6). \\

In concrete terms, the main result of this section, which implies Theorem \ref{thm main} as in \cite{CRS10}, is the following:

\begin{thm}\label{thm improv}
Let $E$ be $(J_s,\rho,d)$-minimal in $B_1$, $0\in\partial E \cap B_1$ \BlueComment{and $\rho$ satisfying A1-A3}. There is a universal $k_0=k_0(n,s,\rho)$ such that if for some $k \geq k_0$  we have a sequence of inclusions
\begin{equation*}
\partial E \cap B_{2^{-l}} \subset \left \{ |x \cdot e_l| \leq 2^{-l} \hat{\rho}(2^{-l}) \right \} \;\;\;\; \forall l \leq k
\end{equation*}
for some sequence of unit vectors $\left \{ e_l\right \}_{l \leq k}$, then we can find unit vectors $e_l$ for all $l \geq k$ such that the inclusions above remain valid for all $l$. \BlueComment{Recall $\hat\rho$ is the auxiliary function defined in Section 2.}
\end{thm}

To understand the asymptotic behavior of these sets in a small ball, we are going to change the scale by a factor of $2^k$ to make the ball $B_{2^{-k}}$ into the unit ball. Then the previous theorem is equivalent to:

\addtocounter{thm}{-1} 
\begin{thm}{\textbf{(Rescaled)}}\label{thm improv prime}
Let $E$ be $(J_s,\rho_k,d_k)$-minimal in $B_{2^k}$ with $\rho_k(r)=\rho(2^{-k}r)$, $d_k=2^kd$, $0\in\partial E$ \BlueComment{ with $\rho$ satisfying A1-A3}. There is a universal $k_0=k_0(n,s,\rho)$ such that if for some $k \geq k_0$, the set $E$ is $a$-flat at $0$ with respect to $\hat\rho$, $a=\hat\rho(2^{-k})$, then we can find a unit vector $e_{\BlueComment{-1}}$ such that we have the inclusion

\begin{equation}\label{eqn main thm 2}
\partial E \cap B_{1/2} \subset \left \{ |x \cdot e_{-1}| \leq 2^{-1}\hat{\rho}(2^{-1-k}) \right \}
\end{equation}
\end{thm}

The proof of the theorem will be divided in a couple of lemmas, all of which deal with a sequence of sets $\{E_k\}$. Here for each $k>0$ the set $E_k$ is $(J_s,\rho_k,d_k)$-minimal as above and each being $a_k$-flat with respect to $\hat\rho$, $a_k=\hat\rho(2^{-k})$.

In order to normalize things further we also assume for all $k$ that $e^{(k)}_0$ agrees with the unit vector in the positive direction of the $x_n$-axis (we can always get in this situation via a rotation).\\

To the sequence $\{E_k\}_k$ we can apply directly the partial Harnack estimate from Theorem \ref{thm Harnack estimate}, as the classical oscillation lemma, it gives a H\"older estimate:


\begin{cor}\label{cor Holder}
There is a universal $k_0$ such that if $k>k_0$, then for any $x \in \partial E_k$, the set $\partial E_k \cap B_\delta(x)$ can be trapped in between the graphs 

\begin{equation}\label{eqn Holder}
\{(y',y_n): y_n=x_n \pm Ca_k \max\{a_k^\gamma, |y'-x'|^\gamma\} \},
\end{equation}

where $C>0$ and $\gamma \in (0,1)$ are universal constants.
\end{cor}

\begin{proof}
The proof is exactly the same as the proof of H\"older continuity for solutions of elliptic equations via the Harnack inequality. That is, if $v$ solves an elliptic equation in $B_1$ then applying  Harnack inequality $k$ times we see that

\begin{equation*}
\sup \limits_{B_{1/2^k}} v-\inf \limits_{B_{1/2^k}} v \leq \mu \left ( \sup \limits_{B_{1/2^{k-1}}} v-\inf \limits_{B_{1/2^{k-1}}} v\right ), \; \mu\in(0,1).
\end{equation*}

Since $k$ is arbitrary, one concludes the oscillation of the function decays geometrically and thus $v$ is $C^{\gamma}$ at $0$ with $\gamma= \log_{2} \mu$. The difference is that in our case we can only apply the Harnack estimate only as long as the flatness assumption holds (thus it is a ``partial'' Harnack estimate). More specifically, we can apply partial Harnack for a $k$-th time as long as $a\mu^k\leq a_0$, where $\mu=1-\delta_0^2$ ($\delta_0$ as in Theorem \ref{thm Harnack estimate}), thus the maximum number of times we can do it for the set $E_k$ is $\sim \log(\frac{a_0}{a_k})$ and from here the H\"older estimate follows.

\end{proof}

This corollary is the key in proving the following lemma, in which we show that from any sequence of vertical dilations of $E_k$, one may always pick a subsequence converging to the graph of continuous function, for a function that does not grow too much at infinity.


\begin{lem}\label{lem blow up}
For each $k$ we dilate the set $E_k$ vertically, and define 

\begin{equation*}
E^*_k= \{ (x',x_n):(x', a_k^{-1}x_n) \in E_k \}, 
\end{equation*}

then, considering the new sequence $\{E^*_k\}_k$, we have:

\begin{enumerate}

\item Along a subsequence $\{E^*_k\}_k$ converges uniformly in compact sets of $\mathbb{R}^n$ to the subgraph of a H\"older continuous function $u:\mathbb{R}^{n-1} \to \mathbb{R}$
\item Moreover, $u(0)=0$ and for a universal $C$ we have 

\begin{equation*}
|u(x')|\leq C (1 + |x'|\hat\rho(|x'|) ), \;\;\; \forall\; x' \in \mathbb{R}^{n-1}.
\end{equation*}

\end{enumerate}

\end{lem}

\begin{proof}

Since the functions $\chi_{E_k}$ are bounded in $H^{s/2}$ there is a subsequence that converges locally in $L^1$ to some set. By the partial H\"older estimate, and the fact that $a_k \to 0$, this convergence is actually uniform in compact sets. Since $0 \in \partial E^*_k$ for all $k$ the same is true for the limit set. Again by the partial H\"older estimate, we see that the limit set can be touched above and below at $0$ by the graphs $\{ x_n=\pm C |x'|^\gamma \}$, its not hard to see after a translation that we may do the same at every point of the boundary of the limit set. Therefore this set is the subgraph of a H\"older continuous function $u:\mathbb{R}^{n-1} \to \mathbb{R}$.

Now we prove the second statement, $u(0)=0$ is clear since $0 \in \partial E$. Moreover, by a diagonalization argument we may assume that for each fixed $l$ that the sequence of vectors $\{e^{(k)}_l\}$ converges to a unit vector $e_l$. If we denote by $p_l \in \mathbb{R}^{n-1}$ the horizontal projection of $e_l$, we can see that 
\begin{equation*}
|u(x')-p_l\cdot x'| \leq C2^l \hat\rho(2^{l}),\;\; \mbox{ if } x'\in B_{2^l},
\end{equation*}
\begin{equation*}
\mbox{ and } |p_l-p_{l+1}|\leq C2^l\hat\rho(2^l),  \;\; \forall\;\; l \geq 0.
\end{equation*}
Therefore, $|u(x')|\leq C(1+|x'|\hat\rho(|x'|))$, $\forall x' \in \mathbb{R}^{n-1}$, which finishes the proof.
\end{proof}

The last lemma we need says basically that when we ``linearize'' the non-local minimal surface operator, we obtain the fractional Laplacian, and thus, the limits of non-homogeneous blow ups must be harmonic.

\begin{lem}\label{lem Liouville}
Let $u$ be the same function from the previous lemma, then 
\begin{equation}\label{eqn harmonic}
\left(-\Delta \right )^{\frac{1+s}{2}}u=0 \mbox{ in } \mathbb{R}^{n-1}.
\end{equation}
\end{lem}

\textbf{Remark}: Given the growth condition of $u$ at infinity,  it follows by estimates for elliptic equations that $u$ is linear (see for instance Landkof's treatise \cite{L72}).

\begin{proof}

We shall prove that (\ref{eqn harmonic}) holds in the viscosity sense (cf. \cite{CS09}). Let $\phi$ be a smooth function, and to fix ideas, suppose it is touching $u$ from below at the origin. Fix positive numbers $\epsilon_1$, $M$ and $K$; Lemma \ref{lem blow up} says that one can find some $k>K$ such that $\partial E_k \cap B_M$  lies in a $a_k\epsilon_1$-neighborhood of the graph of $a_ku(x)$,  $a_k=\hat\rho(2^{-k})$. Moreover, it is easy to see that $\partial E_k$ must be touched from below by a vertical translation of the graph of $a_k\phi(x)$ at some point $x_k \in B_{\epsilon_1}(0)$, \BlueComment{ in particular $\partial E_k$ is being touched at that same point by a ball of radius $R_k\geq C_\phi a_k^{-1}$}.

\BlueComment{For each $k$ let us pick $\epsilon_k$ so that  $a_k^{3/2} \leq \epsilon_k \leq 2a_k^{3/2}$ and such that $\epsilon_k$ is admissible for Theorem \ref{thm second Euler Lagrange}}. We will approximate
\begin{equation*}
L(A_{\epsilon_k},E_k\setminus B_{\delta_k}(\hat{x}))-L(A_{\epsilon_k},E_k^c\setminus B_{\delta_k}(\hat{x}))
\end{equation*}

\noindent with the fractional Laplacian of $u$. In fact, we make the following claim.\\

\noindent \emph{Claim}:
\begin{equation*}
\int_{B'_M(\hat x') \setminus B'_\delta(\hat x')}\frac{u(x')-u(\hat{x}')}{|x'-\hat{x}'|^{n+s}}dx'\leq \frac{1}{2a_k}\int_{B_\delta^c}\frac{\chi_{E_k}(x)-\chi_{E_k^c}(x)}{|x-\hat{x}|^{n+s}}dx+O(\epsilon_1)
\end{equation*}
\begin{equation}\label{eqn harmonic claim}
+O(a_k^{\eta_1})+O(M^{-\eta_2})+C(\phi)\delta^{1-s}, \;\; \mbox{ for universal }\eta_1,\eta_2>0.
\end{equation}

Recalling the proof of Corollary \ref{cor pointwise Euler Lagrange}, we have that 
\begin{equation*}
\int_{B_\delta^c(\hat x)}\frac{\chi_{E_k}(x)-\chi_{E_k^c}(x)}{|x-\hat{x}|^{n+s}}dx \leq C_n\epsilon\delta^{-(1+s)} +\frac{L(A_\epsilon,E_k\setminus B_\delta)-L(A_\epsilon,E_k^c\setminus B_\delta)}{|A_\epsilon|}.
\end{equation*}

\BlueComment{Theorem \ref{thm second Euler Lagrange} guarantees that for all large enough $k$ and fixed small $\delta$,
\begin{equation*}
\int_{B_{\delta_k}^c(\hat x_k)}\frac{\chi_{E_k}(x)-\chi_{E_k^c}(x)}{|x-x_k|^{n+s}}dx \leq  a_k C_0\left ( a_k^{-1}\epsilon_k\delta^{-(1+s)} +C_\phi\delta^{1-s} \right.
\end{equation*}
\begin{equation*}
\left. +(a_k^{-1}\rho(2^{-k}8\epsilon)\epsilon^{-s} \right )=:a_kC_0B_k.
\end{equation*}

In light of this and (\ref{eqn harmonic claim}) to prove the lemma it suffices to get a bound on $B_k$. Observe that from the way we picked $\epsilon_k$  the first term of $B_k$ obviously goes to zero as $k\to \infty$ as long as $\delta$ remains fixed, since
\begin{equation*}
a_k^{-1}\epsilon_k\leq 2a_k^{-1}a_k^{3/2}=2a_k^{1/2} \to 0,
\end{equation*}
the second term in $B_k$ remains bounded in $k$ so we only need to focus on the third one. Once again we apply Lemma \ref{lem rho doubling}, which says that if $k$ is large then
\begin{equation*}
a_k^{-1}\rho(2^{-k}8\epsilon)\epsilon_k^{-s}=\rho(2^{-k}8\epsilon)a_k^{-3s/2}\leq \hat\rho(2^{-k})^m\hat\rho(2^{-k})^{-3s/2},
\end{equation*}
by property A3 (cf. Section 2 ) we always have $m\geq 2$ and we conclude
\begin{equation*}
\lim \limits_{k\to\infty}a_k^{-1}\rho(2^{-k}8\epsilon)\epsilon_k^{-s}=0.
\end{equation*}

In other words, for each fixed $\delta>0$ we have $\lim \limits_{k\to\infty}B_k \leq C_0C_\phi\delta^{1-s}$}. Putting this together with (\ref{eqn harmonic claim}) \BlueComment{we get for $\delta>0$ fixed,}
\begin{equation*}
\int_{B'_M(x_k') \setminus B'_\delta (x_k') }\frac{u(x')-u(x_k')}{|x'-x_k'|^{n+s}}dx' \leq C_0C_\phi \BlueComment{(\delta^{1-s}+o(1))}+O(\epsilon_1)+O(M^{-\eta_2}).
\end{equation*}
Taking $\epsilon_1 \to 0$\BlueComment{, $k\to+\infty$}, we see $x'_k \to 0$, and the inequality above gives us 
\begin{equation*}
\int_{B'_M(0) \setminus B'_\delta (0) }\frac{u(x')-u(0)}{|x'|^{n+s}}dx' \leq C \delta^{1-s}+O(M^{-\eta_2}).
\end{equation*}
Then letting $M \to +\infty$ and $\delta \to 0$ we obtain the desired inequality. The only thing left to prove is the claim (\ref{eqn harmonic claim}) \BlueComment{(which was already proved in \cite{CRS10})}.\\

\emph{Proof of the claim.} Consider for any $\delta>0$ the cylinder

\begin{equation*}
D_m(\hat x ) = \{ |x'-\hat{x}'| \leq m ; \;\; |x_n-(\hat{x})_n|<m \} \subset \mathbb{R}^n.
\end{equation*}

Using the fact that in $D_\delta^c(\hat x)$ the function $|x-\hat{x}|^{-n-s}$ is Lipschitz with constant $C(\delta)$, 

\begin{equation*}
\left | \frac{1}{|x-\hat{x}|^{n+s}}-\frac{1}{|x'-\hat{x}'|^{n+s}}\right | \leq C(\delta)a^2, \;\; \mbox{ whenever } |x_n-(\hat x)_n| \leq  2a.
\end{equation*}

Recalling also that $\partial E_k$ lies in a $a_k\epsilon_1$-neighborhood of the graph of $a_ku(x)$, we get

\begin{equation*}
\int_{D_M(\hat x) \setminus D_\delta (\hat x)}\frac{\chi_{E_k}(x)-\chi_{E_k^c}(x)}{|x-\hat{x}|^{n+s}}dx=2\int_{B'_M (\hat x) \setminus B'_\delta (\hat x)}\frac{a_k \left ( u(x')-u(\hat{x}') + O(\epsilon_1) \right )}{|x'-\hat{x}'|^{n+s}}dx'+O(a_k^3).
\end{equation*}

This is almost estimate (\ref{eqn harmonic claim}), except we still need to take into account the integration over $B_M^c$ and $D_\delta \cap B_\delta^c$, . We use the fact that $\partial E_k$ is being touched from below by the graph of $a_k\phi$, thus we get the lower bound

\begin{equation*}
\int_{D_\delta(\hat x) \setminus B_\delta (\hat x)}\frac{\chi_{E_k}(x)-\chi_{E_k^c}(x)}{|x-\hat{x}|^{n+s}}dx \geq -a_kC(\phi)\delta^{1-s}.
\end{equation*}

Next, note we can neglect the contributions outside $D_M(\hat x)$

\begin{equation*}
\left | \int_{D_M^c(\hat x)}\frac{\chi_{E_k}(x)-\chi_{E_k^c}(x)}{|x-\hat{x}|^{n+s}}dx \right | \leq \int_{M}^\infty a_kr^{n-1}\hat\rho(r)r^{-n-s}dr+Ca_k^{1+\eta_1}
\end{equation*}

\begin{equation*}
\leq Ca_kM^{-\eta_2}+Ca_k^{1+\eta_1}
\end{equation*}

and with this we are done. \end{proof}
\begin{proof}[Proof of Theorem \ref{thm improv prime}]

If \BlueComment{ the } theorem were false, there would be a sequence of sets $\{ E_k \}_{k \geq 0}$ satisfying the hypothesis of the theorem but such that for each $k$, the set $\partial E_k \cap B_{1/2}$ does not lie inside any cylinder of flatness $2^{-1}\hat\rho(2^{-k-1})$. To this sequence we readily apply Lemma \ref{lem blow up} and Lemma \ref{lem Liouville}, so passing to a subsequence we see that the dilated sets $E^*_k$ converge to a hyperplane. Since the convergence is uniform, for any $\epsilon>0$ we can find a $k_1$ for which $\partial E^*_{k_1} \cap B_{1}$ lies inside an $\epsilon$ neighborhood of the limiting plane. 

Taking $\epsilon$ small enough and going back to the set $E_{k_1}$, we see that this set contradicts the assumption that  none of the sets $\partial E_k \cap B_{1/2}$ could be trapped inside  cylinder of flatness $2^{-1}\hat\rho(2^{-k-1})$, and this proves the theorem. \end{proof}


\section{Monotonicity formula and regular points}

In the classical theory of minimal surfaces, an important role is played by the monotonicity formula, which says that whenever $\Sigma$ is a minimal surface:
\begin{equation*}
\frac{d}{dr} \left ( \frac{\mathcal{H}_{n-1}(\Sigma \cap B_r)}{r^{n-1}} \right ) \geq 0,\;\;\; \mbox{ if } 0 \in \Sigma.
\end{equation*}

In this section we consider an \BlueComment{analogous } formula for the non-local case, introduced in \cite{CRS10}. We discuss briefly how it is used to characterize the regular points of the boundary of an almost minimal set $E$, several proofs are omitted.

\begin{DEF}
For any function $u \in L^1(\mathbb{R}^n,\mu)$ where $\mu=\left ( 1 + |x|^2\right )^{-\frac{n+s}{2}}dx$, we define the \textbf{extension} of $u$,  $\hat{u}: \mathbb{R}^{n+1}_+ \to \mathbb{R}$, as the solution to 
\begin{equation*}
\left \{ \begin{array}{rl}
\mbox{div}(z^a\nabla \hat u)=0 & \mbox{ in } \mathbb{R}^{n+1}_+\\
\hat u = u & \mbox{ on } \partial\; \mathbb{R}^{n+1}_+=\mathbb{R}^n
\end{array}	\right.
\end{equation*}

Here $\mathbb{R}^{n+1}_+$ denotes the half space $\{X=(x,z):x\in\mathbb{R}^n,z>0\}$, and $a=1-s$.
\end{DEF}

For more about this extension and its connection to integro-differential equations, see \cite{CS07}.  In \cite{CRS10} it is shown how the non-local energy $J_s(E;\Omega)$ relates to the $H^1(\mathbb{R}^{n+1}_+,z^adxdx)$ energy of $\hat{u}$ when $u=\chi_E-\chi_{E^c}$.

\begin{prop}\label{prop extension energy}{\emph{(See section 7.2 in \cite{CRS10})}}\
Assume $E$ and $F$ are two sets such that $J_s(E;\Omega),J_s(F;\Omega)<\infty$, and suppose that $E \Delta F \subset \subset \Omega$. Then
\begin{equation*}
\BlueComment{\inf \limits_{F}\int_{\Omega^+}z^a\left ( |\nabla \bar{v} |^2-|\nabla \tilde{u}|^2 \right )dxdz = c_{n,s}\left ( J_s(F;\Omega)-J_s(E;\Omega) \right ).}
\end{equation*}
\end{prop}

We will omit the proof of this proposition, since there is nothing to be added to the proof in \cite{CRS10} for it to work in our situation. With this fact in hand we can prove the monotonicity formula for almost minimal sets.

\begin{lem}{(Monotonicity formula)}\label{lem monotonicity formula}
Suppose $E$ is $(J_s,\rho,d)$-minimal in $\Omega$, and that $0 \in \partial E$. Let $u_E= \chi_E-\chi_{E^c}$, then the function
\begin{equation}\label{eqn monotonicity}
\BlueComment{\Phi_E(r)=\frac{\int_{B_r^+}z^a|\nabla \hat{u}_E|^2dX}{r^{n-s}}+(n-s)\int_0^r\rho(t)t^{n-s-1}dt}
\end{equation}
is monotone increasing for all small enough $r$. Moreover, this function is constant if and only if $E$ is a cone with vertex at $0$ and $\rho \equiv 0$.
\end{lem}
\begin{proof}
We can compute $\Phi'_E(r)$ in the standard way to get
\BlueComment{\begin{equation*}
\Phi'_E(r)=\frac{1}{r^{n-s}}\left (\int_{\partial B_r^+}z^a|\nabla \hat{u}_E|^2dX-\frac{(n-s)}{r}\int_{B_r^+}z^a|\nabla \hat{u}_E|^2dX \right )+(n-s)\rho(r)r^{n-s-1}.
\end{equation*}}

The theorem will follow at once from
\begin{equation}\label{eqn monotonicity theorem}
\BlueComment{\frac{r}{n-s}\int_{\partial B_r^+}z^a|\nabla \hat{u}_E|^2dX \geq \int_{B_r^+}z^a|\nabla \hat{u}_E|^2dX-\rho(r)r^{n-s}.}
\end{equation}

To prove (\ref{eqn monotonicity theorem}), consider the function
\BlueComment{\begin{equation*}
\hat{v}(X)= \left \{ 
\begin{array}{ll} 
\hat{u}(r\hat X) & \mbox{ if } |X|\leq r\\
\hat{u}(X) & \mbox{ if } |X| > r
\end{array}\right.,
\end{equation*}}
\BlueComment{where we are putting $\hat X = X/|X|$}. Note that the trace of \BlueComment{$\hat{v}$ } on $\left \{ z=0 \right \}$ equals $\chi_F(x)-\chi_{F^c}(x)$ where $F$ is some set which agrees with $E$ outside the ball $B'_r(0) \subset \mathbb{R}^n$. Then, by Proposition \ref{prop extension energy} and the almost minimality of $E$ we have
\begin{equation}\label{eqn monotonicity theorem 2}
\int_{B_r^+}z^a|\nabla \hat{v}|^2dX \geq \int_{B_r^+}z^a|\nabla \hat{u}|^2dX-\rho(r)r^{n-s}.
\end{equation}
\noindent Using the construction of \BlueComment{$\hat{v}$}, we compute:
\BlueComment{
\begin{equation*}
\int_{B_r^+}z^a|\nabla \hat{v}|^2dX=\int_0^r\int_{\partial B_1^+} (tz)^a|(	\nabla \hat v)(t\hat X)|^2t^n\;dS\;dt
\end{equation*}
\begin{equation*}
\nabla \hat v (X) =r|X|^{-1}\left ( I-\hat X \otimes \hat X \right )(\nabla u_E) (r\hat X)=r|X|^{-1}(\nabla_S u_E) (r\hat X),
\end{equation*}
in the last identity $\nabla_S$ is the component of the gradient tangent to the sphere. Then,
\begin{equation*}
\int_{B_r^+}z^a|\nabla \hat{v}|^2dX=r^2\int_0^rt^{n+a-2}\int_{\partial B_1^+} z^a|(\nabla_S u_E)(r\hat X)|^2\;dS\;dt
\end{equation*}
\begin{equation*}
=r^2\int_0^rt^{n+a-2}dt \int_{\partial B_1^+} z^a|(\nabla_S u_E)(r\hat X)|^2\;dS=\frac{r^{n+2-s}}{n-s} \int_{\partial B_1^+} z^a|(\nabla_S u_E)(r\hat X)|^2\;dS.
\end{equation*}

Changing variables in the last integral, we get
\begin{equation*}
\int_{B_r^+}z^a|\nabla \hat{v}|^2dX=\frac{r}{n-s} \int_{\partial B_r^+} z^a|(\nabla_S u_E)(X)|^2\;dS.	
\end{equation*}

This together with (\ref{eqn monotonicity theorem 2})  implies (\ref{eqn monotonicity theorem}) and the monotonicity is proved. Actually, we have proved the sharper estimate
\begin{equation*}
\Phi_E'(r) \geq  \frac{1}{r^{n-s}}\left (\int_{\partial B_r^+}z^a(\hat{u}_\nu)^2dX\right ).
\end{equation*}
}

From here it also follows that if $\Phi'_E(r)\equiv 0$ then $u_\nu \equiv 0$ on every $\partial B_r$, which shows that $u$ must be homogeneous of degree $0$, in particular it must be so when restricted to $\{ z = 0\}$, but there $u=\chi_E-\chi_{E^c}$, so $E$ must be a cone with vertex at the origin. In this case by the scaling of the energy we see that $\rho \equiv 0$.
\end{proof}

With the monotonicity formula in hand, one can reproduce easily the blow up analysis of classical minimal surfaces (see \cite{G84} for the classical case). The main steps are summarized in the following theorem whose proof we omit. It is worth mentioning that here one needs the results at the end of section 4.

\begin{thm}

Let $E$ be \BlueComment{a $(J_s,\rho,d)$-minimal } in $B_1$, and $0\in \partial E$.

\begin{enumerate}

\item Suppose $r_k \to 0$ is a sequence such that
\begin{equation*}
E_{r_k} \to C \mbox{ in } L^1_{loc}
\end{equation*}
then $C$ is a \emph{minimal cone} with vertex at $0$. Moreover, the blow up sequence $E_{r_k}$ converges locally uniformly to the minimal cone $C$.

\item We have an \textbf{energy gap}: there exists a constant $\delta_0$ such that if $C$ is a minimal cone which is not a half-space, then $\Phi_C \geq \Phi_H +\delta_0$, where $H$ is any half-space.

\item If the minimal cone $C$ is a half-space, then $\partial E$ is $C^{1,\hat\rho}$ in some neighborhood of $0$. In particular, if $\delta_0$ is as in (2) and $\Phi_E(0) < \Phi_H+\delta_0$ then $\partial E$ is smooth in a neighborhood of $0$.

\end{enumerate}
\end{thm}

Finally, let us mention that the dimension reduction analysis done in \cite{CRS10} directly applies to our situation. Thus the estimate achieved there for the dimension of the singular case carries over to almost minimal sets too. Therefore, thanks to Theorem \ref{thm main} we have:

\begin{THM}
If $E$ is $(J_s,\rho,d)$-minimal in $\Omega$, then the Hausdorff dimension of the singular set $\Sigma_E \subset \partial E \cap \Omega$ is at most $n-2$.
\end{THM}

Additionally, the monotonicity formula and the theorem above show that, in the case of the obstacle problem (see Section 3), all points of $E$ near the contact set are regular, specifically, we have the following result:

\begin{THM}
Let $L \subset \subset \Omega$ be a domain with a smooth boundary. Suppose $E$ minimizes $J_s(.;\Omega)$ among all sets containing $L$, then $\partial E$ is $C^{1,\alpha}$ in a neighborhood of $L$ for some $\alpha<s$.
\end{THM}

\section*{Acknowledgements}
The authors would like to thank Luis Caffarelli, Ovidiu Savin and Jean Michel Roquejoffre for several discussions regarding their work and for their interest in this project. The second author was partially supported by a Graduate School Continuing Fellowship from the University of Texas a Austin and NSF Grant DMS-0654267. He would also like to thank his thesis advisor Luis Caffarelli for guidance and support. \BlueComment{The authors are also very much indebted to the anonymous referee for pointing out several important gaps in a first version of the manuscript, in particular in the proofs of Theorem \ref{thm second Euler Lagrange} and Lemma \ref{lem monotonicity formula}.}


\nocite{A76,G84,S07}

\bibliographystyle{amsplain}
\bibliography{nonlocalmc}

\providecommand{\bysame}{\leavevmode\hbox to3em{\hrulefill}\thinspace}
\providecommand{\MR}{\relax\ifhmode\unskip\space\fi MR }
\providecommand{\MRhref}[2]{%
  \href{http://www.ams.org/mathscinet-getitem?mr=#1}{#2}
}
\providecommand{\href}[2]{#2}
\begin{thebibliography}{10}

\bibitem{A76}
F.~J. Almgren, Jr., \emph{Existence and regularity almost everywhere of
  solutions to elliptic variational problems with constraints}, Mem. Amer.
  Math. Soc. \textbf{4} (1976), no.~165, viii+199. \MR{MR0420406 (54 8420)}

\bibitem{CS07}
Luis Caffarelli and Luis Silvestre, \emph{An extension problem related to the
  fractional {L}aplacian}, Comm. Partial Differential Equations \textbf{32}
  (2007), no.~7-9, 1245--1260. \MR{MR2354493 (2009k:35096)}

\bibitem{CS09}
\bysame, \emph{Regularity theory for fully nonlinear integro-differential
  equations}, Comm. Pure Appl. Math. \textbf{62} (2009), no.~5, 597--638.
  \MR{MR2494809 (2010d:35376)}

\bibitem{CC93}
Luis~A. Caffarelli and Antonio C{\'o}rdoba, \emph{An elementary regularity
  theory of minimal surfaces}, Differential Integral Equations \textbf{6}
  (1993), no.~1, 1--13. \MR{MR1190161 (94c:49042)}

\bibitem{CRS10}
Luis~A. Caffarelli, Jean-Michel Roquejoffre, and Ovidiu Savin, \emph{Non-local
  minimal surfaces}, to appear in Communications on Pure and Applied
  Mathematics (2010).

\bibitem{CS10}
Luis~A. Caffarelli and Panagiotis Souganidis, \emph{Convergence of nonlocal
  threshold dynamics approximations to front propagation}, Archive for Rational
  Mechanics and Analysis \textbf{195} (2010), no.~1, 1--23.

\bibitem{dG06}
Ennio De~Giorgi, \emph{Selected papers}, Springer-Verlag, Berlin, 2006, Edited
  by Luigi Ambrosio, Gianni Dal Maso, Marco Forti, Mario Miranda and Sergio
  Spagnolo. \MR{MR2229237 (2007d:49001)}

\bibitem{G84}
Enrico Giusti, \emph{Minimal surfaces and functions of bounded variation},
  Monographs in Mathematics, vol.~80, Birkh\"auser Verlag, Basel, 1984.
  \MR{MR775682 (87a:58041)}

\bibitem{I09}
Cyril Imbert, \emph{Level set approach for fractional mean curvature flows},
  Interfaces Free Bound. \textbf{11} (2009), no.~1, 153--176. \MR{MR2487027}

\bibitem{L72}
N.~S. Landkof, \emph{Foundations of modern potential theory}, Springer-Verlag,
  New York, 1972, Translated from the Russian by A. P. Doohovskoy, Die
  Grundlehren der mathematischen Wissenschaften, Band 180. \MR{MR0350027 (50
  \#2520)}

\bibitem{L90}
Stephan Luckhaus, \emph{Solutions for the two-phase {S}tefan problem with the
  {G}ibbs-{T}homson law for the melting temperature}, European J. Appl. Math.
  \textbf{1} (1990), no.~2, 101--111. \MR{MR1117346 (92i:80004)}

\bibitem{LS95}
Stephan Luckhaus and Thomas Sturzenhecker, \emph{Implicit time discretization
  for the mean curvature flow equation}, Calc. Var. Partial Differential
  Equations \textbf{3} (1995), no.~2, 253--271. \MR{MR1386964 (97e:65085)}

\bibitem{S07}
Ovidiu Savin, \emph{Small perturbation solutions for elliptic equations}, Comm.
  Partial Differential Equations \textbf{32} (2007), no.~4-6, 557--578.
  \MR{MR2334822 (2008k:35175)}

\bibitem{S09}
\bysame, \emph{Regularity of flat level sets in phase transitions}, Ann. of
  Math. (2) \textbf{169} (2009), no.~1, 41--78. \MR{MR2480601 (2009m:58025)}

\bibitem{T84}
Italo Tamanini, \emph{Regularity results for almost minimal oriented
  hypersurfaces in $\mathbb{R}^n$}, Quaderni del Dipartimento di Matematica
  dell'Universitˆ degli Studi di Lecce, 1984.

\end{thebibliography}

\end{document}